\documentclass[12pt,reqno,a4paper]{amsart}
\usepackage[english]{babel}
\usepackage[applemac]{inputenc}
\usepackage[T1]{fontenc}
\usepackage{palatino}
\usepackage{amsfonts}
\usepackage{amsmath}
\usepackage{amssymb}
\usepackage{amsthm}
\usepackage{verbatim}
\usepackage{calrsfs}
\usepackage{graphicx}
\usepackage[colorlinks = true, citecolor = black]{hyperref}
\newcommand{\R}{\mathbb{R}}

\newcommand{\N}{\mathbb{N}}
\newcommand{\C}{\mathbb{C}}
\newcommand{\Z}{\mathbb{Z}}
\newcommand{\calO}{\mathcal{O}}

\newcommand{\calH}{\mathcal{H}}

\newcommand{\calL}{\mathcal{L}}

\newcommand{\spt}{\operatorname{spt}}
\newcommand{\Arg}{\operatorname{Arg}}
\newcommand{\interior}{\operatorname{int}}
\newcommand{\card}{\operatorname{card}}

\theoremstyle{plain}
\newtheorem{thm}[equation]{Theorem}
\newtheorem{lemma}[equation]{Lemma}
\newtheorem{proposition}[equation]{Proposition}
\newtheorem{cor}[equation]{Corollary}

\newtheorem{claim}[equation]{Claim}

\theoremstyle{definition}
\newtheorem{definition}[equation]{Definition}
\newtheorem{defThm}[equation]{Theorem and Definition}

\theoremstyle{remark}
\newtheorem{remark}[equation]{Remark}

\numberwithin{equation}{section}

\author{Tuomas Orponen}\thanks{The author is supported by the Finnish Centre of Excellence in Analysis and Dynamics Research.}
\title{Slicing Sets and Measures, and the Dimension of Exceptional Parameters}
\address{Department of Mathematics and Statistics, University of Helsinki, P.O.B. 68, FI-00014 Helsinki, Finland}
\email{tuomas.orponen@helsinki.fi}

\begin{document}

\begin{abstract} We consider the problem of slicing a compact metric space $\Omega$ with sets of the form $\pi_{\lambda}^{-1}\{t\}$, where the mappings $\pi_{\lambda} \colon \Omega \to \R$, $\lambda \in \R$, are \emph{generalized projections}, introduced by Yuval Peres and Wilhelm Schlag in 2000. The basic question is: assuming that $\Omega$ has Hausdorff dimension strictly greater than one, what is the dimension of the 'typical' slice $\pi_{\lambda}^{-1}\{t\}$, as the parameters $\lambda$ and $t$ vary. In the special case of the mappings $\pi_{\lambda}$ being orthogonal projections restricted to a compact set $\Omega \subset \R^{2}$, the problem dates back to a 1954 paper by Marstrand: he proved that for almost every $\lambda$ there exist positively many $t \in \R$ such that $\dim \pi_{\lambda}^{-1}\{t\} = \dim \Omega - 1$. For generalized projections, the same result was obtained 50 years later by J\"arvenp\"a\"a, J\"arvenp\"a\"a and Niemel\"a. In this paper, we improve the previously existing estimates by replacing the phrase 'almost all $\lambda$' with a sharp bound for the dimension of the exceptional parameters.
\end{abstract}

\maketitle

\section{Introduction}

\emph{Generalised projections} were introduced by Yuval Peres and Wilhelm Schlag in 2000: these are families of continuous mapping $\pi_{\lambda} \colon \Omega \to \R^{n}$, $\lambda \in Q$, where $\Omega$ is a compact metric space and $Q$ is an open set of parameters in $\R^{m}$. The projections $\pi_{\lambda}$ are required to satisfy certain conditions, see Definition \ref{projections}, which guarantee that they behave regularly with respect to $\lambda$ and are never too severely non-injective. As it was shown in \cite{PS}, these conditions are sufficient to produce results in the spirit of Marstrand's projection theorem, which were previously known only for orthogonal projections in $\R^{n}$.

Let us quickly review the classical theory related with orthogonal projections, restricting attention to $\R^{2}$. According to a 1954 result of Marstrand, see \cite{Mar}, any Borel set $B$ of dimension $\dim B = s \leq 1$ is projected into a set of dimension $s$ in almost all directions; if $\dim B > 1$, the projections typically have positive length. Since 1954, these results have been sharpened by examining the largest possible dimension of the set of \emph{exceptional} directions, that is, the directions for which the typical behaviour described by the theorem fails. Let $\rho_{\theta}(x) = x \cdot (\cos \theta,\sin \theta)$ denote the orthogonal projection onto the line spanned by $(\cos \theta, \sin \theta)$. Then we have the bounds
\begin{equation}\label{kaufman} \dim \{\theta \in [0,2\pi) : \dim \rho_{\theta}(B) < \dim B\} \leq \dim B, \qquad \dim B \leq 1, \end{equation} 
due to Kaufman, see \cite{Ka}, and
\begin{equation}\label{falconer} \dim \{\theta \in [0,2\pi) : \calL^{1}(\rho_{\theta}(B)) = 0\} \leq 2 - \dim B, \qquad \dim B > 1,  \end{equation} 
due to Falconer, see \cite{Fa}. Both estimates are known to be sharp. 

Orthogonal projections are a special case of the general formalism of Peres and Schlag, and all the results stated above -- along with their higher dimensional analogues -- follow from the theory in \cite{PS}. Besides orthogonal projections, Peres and Schlag provide multiple examples to demonstrate the wide applicability of their formalism. These examples include the mappings $\pi_{\lambda}(x) = |x - \lambda|^{2}$, $\lambda,x \in \R^{n}$, and the 'Bernoulli projections'
\begin{displaymath} \pi_{\lambda}(x) = \sum_{j = 1}^{\infty} x_{j}\lambda^{j}, \qquad x = (x_{1},x_{2},\ldots) \in \Omega = \{-1,1\}^{\N}, \end{displaymath} 
for $\lambda \in (0,1)$. Indeed, estimates similar to \eqref{kaufman} and \eqref{falconer} are obtained for these projections, and many more, in \cite{PS}.

In the field of geometric measure theory, Marstrand's projection theorem is not the only result involving orthogonal projections. Marstrand's theorem is certainly matched in fame by the \emph{Besicovitch-Federer projection theorem}, characterising rectifiability in $\R^{n}$ in terms of the behaviour of orthogonal projections. In the plane, this theorem states that a Borel set $B$ with positive and finite $1$-dimensional measure is purely unrectifiable, if and only if almost all of the sets $\rho_{\theta}(B)$ have zero length. Considering the success of Peres and Schlag's projections in generalizing Marstrand's theorem, it is natural to ask whether also the characterisation of Besicovitch-Federer would permit an analogue in terms of the generalized projections. This question was recently resolved by Hovila, J\"arvenp\"a\"a, J\"arvenp\"a\"a and Ledrappier: Theorem 1.2 in \cite{HJJL} essentially shows that orthogonal projections can be replaced by \emph{any} family of generalized projections in the theorem of Besicovitch-Federer.

There is a third classical result in geometric measure theory related intimately, though slightly covertly, to orthogonal projections. In his 1954 article mentioned above, Marstrand also studied the following question: given a Borel set $B \subset \R^{2}$ with $\dim B > 1$, what can be said of the dimension of the intersections $B \cap L$, where $L$ ranges over the lines of $\R^{2}$? Marstrand proved that in almost all directions there exist positively many lines intersecting $B$ in a set of dimension $\dim B - 1$. The multidimensional analogue of this result was obtained by Mattila first in \cite{Mat1} and later in \cite{Mat2} using a different technique: if $B \subset \R^{n}$ is a Borel set of dimension $\dim B > m$, then positively many translates of almost every $m$-codimensional subspace intersect $B$ in dimension $\dim B - m$. These results are easily formulated in terms of orthogonal projections. Once more restricting attention to the plane, the theorem of Marstrand can be stated as follows: given a Borel set $B \subset \R^{2}$ with $\dim B > 1$, we have
\begin{equation}\label{marstrand2} \calL^{1}(\{t \in \R : \dim [B \cap \rho_{\theta}^{-1}\{t\}] \geq \dim B - 1\}) > 0 \end{equation} 
for $\calL^{1}$ almost every $\theta \in [0,2\pi)$. Indeed, all lines of $\R^{2}$ are of the form $\rho_{\theta}^{-1}\{t\}$ for some $\theta \in [0,2\pi)$ and $t \in \R$. Inspecting \eqref{marstrand2}, one arrives at the following conjecture: let $J \subset \R$ be an open interval and let $\pi_{\lambda} \colon \Omega \to \R$ be a family of generalized projections. Then \eqref{marstrand2} holds with $\rho_{\theta}$ replaced with $\pi_{\lambda}$, for any Borel set $B \subset \Omega$ with $\dim B > 1$. This conjecture, and its higher dimensional analogue, was verified in the 2004 paper \cite{JJN} by J\"arvenp\"a\"a, J\"arvenp\"a\"a and Niemel\"a. 

To the best of our knowledge, no work has previously been done on obtaining a sharpened version of \eqref{marstrand2}. By a 'sharpened version' we mean a result, which would yield \eqref{marstrand2} not only for almost all $\theta \in [0,2\pi)$, but also give an estimate for the dimension of the set of exceptional parameters $\theta$ similar to \eqref{falconer}. It is easy to guess the correct analogue of \eqref{falconer} in our situation: first, note that if \eqref{marstrand2} holds for some $\theta \in [0,2\pi)$, then we automatically have $\calL^{1}(\rho_{\theta}(B)) > 0$. The estimate \eqref{falconer} is known sharp, which means that $\calL^{1}(\rho_{\theta}(B)) > 0$ may fail (for some particular set $B$, see Section \ref{sharpness} for references) for all parameters $\theta$ in a set of dimension $2 - \dim B$: thus, also \eqref{marstrand2} may fail for all parameters $\theta$ in a set of dimension $2 - \dim B$. The converse result is proven below: for any Borel set $B \subset \R^{2}$ with $\dim B > 1$, we have \eqref{marstrand2} for all parameters $\theta \in [0,2\pi) \setminus E$, where $\dim E \leq 2 - \dim B$. Inspired by the generalization due to J\"arvenp\"a\"a, J\"arvenp\"a\"a and Niemel\"a, all our estimates will also be couched in the formalism of Peres and Schlag.
Straightforward applications to the mappings $\pi_{\lambda}(x) = |\lambda - x|^{2}$ and $\pi_{\lambda}(x) = \sum x_{i}\lambda^{i}$ are presented in Sections \ref{sharpness} and \ref{further}.

\section{Definitions and a Result of Peres and Schlag}

We start by defining our central object of study, the projections $\pi_{\lambda}$:

\begin{definition}[The Projections $\pi_{\lambda}$]\label{projections} Let $(\Omega,d)$ be a compact metric space. Suppose that an open interval $J \subset \R$ parametrises a collection of continuous mappings $\pi_{\lambda} \colon \Omega \to \R$, $\lambda \in J$. These mappings, often referred to as \emph{projections}, are assumed to satisfy the following properties (see Remark \ref{ps} for a discussion on the origins of our assumptions):
\begin{itemize}
\item[(i)] For every compact subinterval $I \subset J$ and every $l \in \N$ there exist constants $C_{I,l} > 0$ such that
\begin{displaymath} |\partial^{l}_{\lambda} \pi_{\lambda}(x)| \leq C_{I,l} \end{displaymath}
for every $\lambda \in I$ and $x \in \Omega$.
\item[(ii)] Write
\begin{displaymath} \Phi_{\lambda}(x,y) := \begin{cases} \frac{\pi_{\lambda}(x) - \pi_{\lambda}(y)}{d(x,y)}, & \lambda \in J, \: x,y \in \Omega, \: x \neq y\\
											0, & \lambda \in J, \: x = y \in \Omega, \end{cases} \end{displaymath} 
and fix $\tau \in [0,1)$. For every compact subinterval $I \subset J$, there exist constants $\delta_{I,\tau}$ such that
\begin{equation}\label{transversality} |\Phi_{\lambda}(x,y)| \leq \delta_{I,\tau}d(x,y)^{\tau} \quad \Longrightarrow \quad |\partial_{\lambda}\Phi_{\lambda}(x,y)| \geq \delta_{I,\tau}d(x,y)^{\tau} \end{equation} 
for every $\lambda \in I$ and $x,y \in \Omega$. The projections $\pi_{\lambda}$ are then said to satisfy \emph{transversality of order $\tau$}. 
\item[(iii)] For every compact subinterval $I \subset J$, every $\tau > 0$ and every $l \in \N$ there exist constants $C_{I,l,\tau} > 0$ such that
\begin{equation}\label{regularity} |\partial^{(l)}_{\lambda}\Phi_{\lambda}(x,y)| \leq C_{I,l,\tau}d(x,y)^{-l\tau} \end{equation} 
for every $\lambda \in I$ and $x,y \in \Omega$. This property is called \emph{regularity of order $\tau$}.
\end{itemize}
\end{definition}
Under these hypotheses, our main result is the following
\begin{thm}\label{main} Let $(\pi_{\lambda})_{\lambda \in J} \colon \Omega \to \R$ be a family of projections satisfying (i), (ii) and (iii) of Definition \ref{projections} for some $\tau > 0$. Let $B \subset \Omega$ be a Borel set with $\dim B = s$ for some $1 < s < 2$. Then there exists a set $E \subset J$ such that $\dim E \leq 2 - s + \delta(\tau)$, and 
\begin{displaymath} \calL^{1}(\{t \in \R : \dim [B \cap \pi_{\lambda}^{-1}\{t\}] \geq s - 1 - \delta(\tau)\}) > 0, \qquad \lambda \in J \setminus E. \end{displaymath}
Here $\delta(\tau) > 0$ is a constant depending only on $\tau$, and $\delta(\tau) \to 0$ as $\tau \to 0$. If the requirements of Definition \ref{projections} are satisfied with $\tau = 0$ and $\calH^{s}(B) > 0$, then the assertions above hold with $\delta(\tau) = 0$.
\end{thm}

\begin{remark}\label{ps} Throughout the paper, $\dim$ will always refer to Hausdorff dimension. Our definition of the projections $\pi_{\lambda}$ is a slightly specialized version of \cite[Definition 2.7]{PS}. The most notable strengthening in our hypotheses is that, in the original definition in \cite{PS}, the bound in (\ref{regularity}) is only assumed to hold \emph{under the condition that} $|\Phi_{\lambda}(x,y)| \leq \delta_{I,\tau}d(x,y)^{\tau}$, whereas we assume it for all $\lambda \in I$ and $x,y \in \Omega$. Second, Peres and Schlag also obtain results for projections $\pi_{\lambda}$ such that (\ref{regularity}) holds only for a finite number of $\lambda$-derivatives of the function $\Phi_{\lambda}$: our projections are $\infty$-regular in the language of \cite{PS}.

Inspecting (i) and (ii) above, one should note that the easiest way to establish (\ref{transversality}) and (\ref{regularity}) for all $\tau > 0$ is to establish them (with some constants) for $\tau = 0$: indeed, this is possible in all known (to the author, at least) 'geometric' applications of the projection formalism -- but not possible in \emph{all} applications. 
\end{remark}
A word on notation before we proceed. If $A,B > 0$ and $p_{1},\ldots,p_{k}$ are parameters, we write $A \lesssim_{p_{1},\ldots,p_{k}} B$, if there exists a finite constant $C > 0$, depending only on the parameters $p_{1},\ldots,p_{k}$, such that $A \leq CB$. The two-sided inequality $A \lesssim_{p_{1},\ldots,p_{k}} B \lesssim_{p_{1},\ldots,p_{k}} A$ is further abbreviated to $A \asymp_{p_{1},\ldots,p_{k}} B$. 

We now cite the parts of \cite[Theorem 2.8]{PS} that will be needed later:

\begin{thm}\label{sobolev} Let $\pi_{\lambda} \colon \Omega \to \R$, $\lambda \in J$, be a family of projections as in Definition \ref{projections}, satisfying (ii) and (iii) for some $\tau \in [0,1)$. Let $\mu$ be a Radon measure on $\Omega$ such that $I_{s}(\mu) := \iint d(x,y)^{-s} \, d\mu x \, d\mu y < \infty$ for some $s > 0$. Write $\mu_{\lambda} := \pi_{\lambda\sharp}\mu$.\footnote{Thus $\mu_{\lambda}(B) = \mu(\pi_{\lambda}^{-1}(B))$ for Borel sets $B \subset \R$.} Then
\begin{equation}\label{sobolevIneq} \int_{I} \|\mu_{\lambda}\|_{2,\gamma}^{2} \, d\lambda \lesssim_{I,\gamma} I_{s}(\mu), \qquad 0 < (1 + 2\gamma)(1 + a_{0}\tau) \leq s, \end{equation}
where $\|\mu_{\lambda}\|_{2,\gamma}^{2} := \int |t|^{2\gamma}|\widehat{\mu_{\lambda}}(t)|^{2} \, dt$ is the Sobolev-norm of $\mu_{\lambda}$ with index $\gamma \in \R$, and $a_{0} > 0$ is an absolute constant. Moreover, we have the estimate
\begin{displaymath} \dim \{\lambda \in J : \mu_{\lambda} \not\ll \calL^{1}\} \leq 2 - \frac{s}{1 + a_{0}\tau}. \end{displaymath}
\end{thm}

\begin{definition}[The Mappings $\Psi_{\lambda}$]\label{psi} Given a family of projections as in Definition \ref{projections}, we define the mappings $\Psi_{\lambda} \colon \Omega \to \R^{2}$, $\lambda \in J$, by
\begin{displaymath} \Psi_{\lambda}(x) := (\partial_{\lambda}\pi_{\lambda}(x),\pi_{\lambda}(x)), \qquad x \in \Omega. \end{displaymath}
\end{definition}

\begin{remark}[H\"older continuity of $\Psi_{\lambda}$] Note that
\begin{displaymath} \frac{\Psi_{\lambda}(x) - \Psi_{\lambda}(y)}{d(x,y)} = (\partial_{\lambda}\Phi_{\lambda}(x,y),\Phi_{\lambda}(x,y)), \qquad x,y \in \Omega,\: x \neq y. \end{displaymath} 
Assuming that the projections $\pi_{\lambda}$ are transversal and regular of order $\tau \in [0,1)$, the inequalities (\ref{transversality}) and (\ref{regularity}) yield
\begin{displaymath} \left| \frac{\Psi_{\lambda}(x) - \Psi_{\lambda}(y)}{d(x,y)} \right| \gtrsim |\partial_{\lambda}\Phi_{\lambda}(x,y)| + |\Phi_{\lambda}(x,y)| \geq \delta_{I,\tau}d(x,y)^{\tau} \end{displaymath}
and 
\begin{displaymath} \left| \frac{\Psi_{\lambda}(x) - \Psi_{\lambda}(y)}{d(x,y)} \right| \lesssim \max\{C_{I,0,\tau},C_{I,1,\tau}\}d(x,y)^{-\tau} \end{displaymath} 
for all $\lambda$ in any compact subinterval $I \subset J$ and all $x,y \in \Omega$. In brief, the mapping $\Psi_{\lambda}$ satisfies bi-H\"older continuity in the form
\begin{equation}\label{holder} d(x,y)^{1 + \tau} \lesssim_{I,\tau} |\Psi_{\lambda}(x) - \Psi_{\lambda}(y)| \lesssim_{I,\tau} d(x,y)^{1 - \tau} \end{equation}
for $x,y \in \Omega$.
\end{remark}
We close this chapter by stating a result on 'slicing' any Radon measure $\mu$ on $\R^{2}$ with respect to a continuous function $\pi \colon \R^{2} \to \R$. For the technical details, we refer to \cite{Mat2} or \cite[Chapter 10]{Mat3}. 

\begin{defThm}[Sliced Measures]\label{slices} Let $\mu$ be a compactly supported Radon measure on $\R^{2}$. If $\rho \colon \R^{2} \to \R$ is any orthogonal projection, we may for $\calL^{1}$ almost every $t \in \R$ define the \emph{sliced measure} $\mu_{\rho,t}$ with the following properties:
\begin{itemize}
\item[(i)] $\spt \mu_{\rho,t} \subset \spt \mu \cap \rho^{-1}\{t\}$,
\item[(ii)] 
\begin{displaymath} \int \eta \, d\mu_{\rho,t} = \lim_{\delta \to 0} (2\delta)^{-1} \int_{\rho^{-1}(t - \delta,t + \delta)} \eta \, d\mu, \qquad \eta \in C(\R^{2}). \end{displaymath}
\end{itemize}
For every Borel set $B \subset \R$ and non-negative lower semicontinuous function $\eta$ on $\R^{2}$ we have the inequality
\begin{displaymath} \int_{B} \int \eta \, d\mu_{\rho,t} \, dt \leq \int_{\rho^{-1}(B)} \eta \, d\mu. \end{displaymath}
Moreover, equality holds, if $\rho_{\sharp}\mu \ll \calL^{1}$. In this case, taking $B = \{t : \nexists \,\mu_{\rho,t} \text{ or } \mu_{\rho,t} \equiv 0\}$ and $\eta \equiv 1$ yields
\begin{displaymath} \rho_{\sharp}\mu(B) = \int_{\rho^{-1}(B)} \, d\mu = \int_{B} \int d\mu_{\rho,t} \, dt = 0, \end{displaymath}
which shows that $\mu_{\rho,t}$ exists and is non-trivial for $\rho_{\sharp}\mu$ almost every $t \in \R$. 
\end{defThm}

\section{Preliminary Lemmas}

For the rest of the paper we will write $e_{1} = (1,0)$, $e_{2} = (0,1)$, $\rho_{1} := \rho_{e_{1}}$ and $\rho_{2} := \rho_{e_{2}}$. Thus $\rho_{1}(x,y) = x$ and $\rho_{2}(x,y) = y$ for $(x,y) \in \R^{2}$. Our first lemma is motivated by the following idea: if $\mu$ were a smooth function on $\R^{2}$, say $\mu \in C_{c}^{\infty}(\R^{2})$ (the subscript $c$ indicates compact support), then one may easily check that any slice $\mu_{t} := \mu_{\rho_{2},t}$, $t \in \R$, coincides with $(\mu\calH^{1})\llcorner L_{t}$, where $L_{t}$ is the line $L_{t} = \{(s,t) : s \in \R\}$, and $(\mu\calH^{1})\llcorner L_{t}$ is the measure defined by
\begin{displaymath} \int \varphi \, d(\mu\calH^{1})\llcorner L_{t} := \int_{\R} \varphi(s,t)\mu(s,t) \, ds \end{displaymath}
for $\varphi \in C(\R^{2})$.  Since an arbitrary Radon measure $\mu$ on $\R^{2}$ can be approximated by a family $(\mu_{\varepsilon})_{\varepsilon > 0}$ of smooth functions, one may ask whether also the measures $(\mu^{\varepsilon}\calH^{1})\llcorner L_{t}$ converge weakly to $\mu_{t}$ as $\varepsilon \searrow 0$. Below, we will prove that \textbf{if the functions $\mu_{\varepsilon}$ are chosen suitably}, then $\mu_{t} = \lim_{\varepsilon \to 0} (\mu_{\varepsilon}\calH^{1})\llcorner L_{t}$.

\begin{lemma} Let $Q := (-1/2,1/2) \times (-1/2,1/2) \subset \R^{2}$ be the unit square centered at the origin, and let $\chi_{\varepsilon}$ be the lower semicontinuous function $\chi_{\varepsilon}(x) := \varepsilon^{-2}\chi_{Q}(x/\varepsilon)$. Let $\mu$ be a compactly supported Radon measure on $\R^{2}$, and write $\mu_{\varepsilon} := \chi_{\varepsilon} \ast \mu$. Then
\begin{equation}\label{form1} \int_{\R^{2}} \eta \, d\mu_{t} = \lim_{\varepsilon \to 0} \int_{\R} \eta(s,t)\mu_{\varepsilon}(s,t) \, ds, \qquad \eta \in C(\R^{2}), \end{equation}
whenever $\mu_{t} := \mu_{\rho_{2},t}$ exists in the sense of Definition \ref{slices}. Moreover, the convergence is uniform on any compact family (in the $\sup$-norm topology) of functions $K \subset C(\R^{2})$.
\end{lemma}

\begin{proof} We first establish pointwise convergence, and then use the Arzelà-Ascoli theorem to verify the stronger conclusion. Assume that $\mu_{t}$ exists for some $t \in \R$. We start with a simple reduction. Namely, we observe that it suffices to prove \eqref{form1} only for functions $\eta$ with the special form $\eta(x_{1},x_{2}) = \eta(x_{1},t)$, $(x_{1},x_{2}) \in \R^{2}$. Indeed, if $\eta \in C(\R^{2})$ is arbitrary, we note that both sides of \eqref{form1} depend only on the values of $\eta$ on the line $L_{t}$: in particular, both sides of \eqref{form1} remain unchanged, if we replace $\eta$ by the function $\tilde{\eta} \in C^{+}(\R^{2})$ defined by $\tilde{\eta}(x_{1},x_{2}) = \eta(x_{1},t)$. The function $\tilde{\eta}$ has the 'special form'.

Fix a function $\eta \in C(\R^{2})$ of the 'special form'. 
Starting from the right hand side of \eqref{form1}, we compute
\begin{align}\label{form30} \int_{\R} & \eta(s,t)\mu_{\varepsilon}(s,t) \, ds = \int_{\R} \eta(s,t) \varepsilon^{-2} \int_{\R^{2}} \chi_{Q}\left(\frac{(s,t) - x}{\varepsilon}\right) \, d\mu x \, ds\notag\\
& = \varepsilon^{-1} \int_{\rho_{2}^{-1}(t - \varepsilon/2, t + \varepsilon/2)} \left(\varepsilon^{-1} \int_{x_{1} - \varepsilon/2}^{x_{1} + \varepsilon/2} \eta(s,t)\, ds \right) \, d\mu x. \end{align}
The domain $\rho_{2}^{-1}(t - \varepsilon/2,t + \varepsilon/2)$ results from the fact that the kernel $\chi_{Q}([(s,t) - x]/\varepsilon)$ is zero whenever the second coordinate of $x = (x_{1},x_{2})$ differs from $t$ by more than $\varepsilon/2$. On the other hand, if $|t - x_{2}| < \varepsilon/2$, we see that $\chi_{Q}([(s,t) - x]/\varepsilon) = 1$, if and only if $|s - x_{1}| < \varepsilon/2$. Next, we use the uniform continuity of $\eta$ on $\spt \mu$ and the 'special form' property to deduce that 
\begin{align}\label{form2} \sup & \left\{\left|\varepsilon^{-1} \int_{x_{1} - \varepsilon/2}^{x_{1} + \varepsilon/2} \eta(s,t) \, ds - \eta(x) \right| : x = (x_{1},x_{2}) \in \spt \mu \right\}\notag\\
& \leq \sup \left\{\varepsilon^{-1} \int_{x_{1} - \varepsilon/2}^{x_{1} + \varepsilon/2} |\eta(s,t) - \eta(x_{1},t)| \, ds : (x_{1},x_{2}) \in \spt \mu \right\} \to 0, \end{align}
as $\varepsilon \to 0$. We write $\| \cdot \|_{\infty}$ for the $L^{\infty}$-norm on $C(\spt \mu)$.  Let us consider the continuous linear functionals $\Lambda_{\epsilon} \colon (C(\spt \mu),\| \cdot \|_{\infty}) \to \R$, defined by
\begin{displaymath} \Lambda_{\varepsilon}(\psi) := \varepsilon^{-1}\int_{\rho_{2}^{-1}(t - \varepsilon/2,t + \varepsilon/2)} \psi \, d\mu, \qquad \psi \in C(\spt \mu). \end{displaymath}
Since $\mu_{t}$ exists, the orbits $\{\Lambda_{\varepsilon}(\psi) : \varepsilon > 0\}$ are bounded subsets of $\R$ for any $\psi \in C(\spt \mu)$. So, it follows from the Banach-Steinhaus theorem, see \cite[Theorem 2.5]{Ru}, that these functionals are uniformly bounded: there exists $C > 0$, independent of $\varepsilon > 0$, such that $|\Lambda_{\varepsilon}(\psi)| \leq C\|\psi\|_{\infty}$. We apply the bound with $\psi = \psi_{\varepsilon}$ defined by
\begin{displaymath} \psi_{\varepsilon}(x) = \varepsilon^{-1} \int_{x_{1} - \varepsilon/2}^{x_{1} + \varepsilon/2} \eta(s,t) \, ds - \eta(x), \quad x = (x_{1},x_{2}) \in \spt \mu. \end{displaymath}
Recalling \eqref{form2}, we have
\begin{align*} \limsup_{\varepsilon \to 0} & \left| \varepsilon^{-1} \int_{\rho_{2}^{-1}(t - \varepsilon/2, t + \varepsilon/2)} \left[\varepsilon^{-1} \int_{x_{1} - \varepsilon/2}^{x_{1} + \varepsilon/2} \eta(s,t) \, ds -  \eta(x) \right] d\mu x \right|\\
& = \limsup_{\varepsilon \to 0} |\Lambda_{\varepsilon}(\psi_{\varepsilon})| \leq C \cdot \limsup_{\varepsilon \to 0} \|\psi_{\varepsilon}\|_{\infty} \stackrel{\eqref{form2}}{=} 0, \end{align*}
which implies that
\begin{align*} \lim_{\varepsilon \to 0} \varepsilon^{-1} & \int_{\rho_{2}^{-1}(t - \varepsilon/2, t + \varepsilon/2)} \left(\varepsilon^{-1} \int_{x_{1} - \varepsilon/2}^{x_{1} + \varepsilon/2} \eta(s,t) \, ds \right) \, d\mu x\\
& = \lim_{\varepsilon \to 0} \varepsilon^{-1} \int_{\rho_{2}^{-1}(t - \varepsilon/2,t + \varepsilon/2)} \eta(x) \, d\mu x =: \int \eta \, d\mu_{t}. \end{align*}
The existence of the former limit is a consequence of this equation, and the \emph{a priori} information on the existence of the latter limit. Combined with \eqref{form30}, this finishes the proof of pointwise convergence in \eqref{form1}.

Next, we fix a compact family of functions $K \subset C(\R^{2})$, and demonstrate that the convergence is uniform on $K$. Let $B \subset \R^{2}$ be a closed ball large enough to contain the supports of all the measures $\mu_{\varepsilon}$, for $0 < \varepsilon \leq 1$, say. Consider the linear functionals
\begin{displaymath} \Gamma_{\varepsilon}(\psi) := \int_{\R} \psi(r,t)\mu_{\varepsilon}(r,t) \, dr, \qquad \psi \in C(B). \end{displaymath}
Since the functionals $\Gamma_{\varepsilon}$ and $\psi \mapsto \Gamma(\psi) := \int \psi \, d\mu_{t}$ vanish outside $C(B)$, it suffices to show that $\Gamma_{\epsilon} \to \Gamma$ uniformly on $K \cap C(B)$: in fact, we may and will assume that $K \subset C(B)$. This way, we may view the mappings $\Gamma_{\varepsilon}$ not only as a family of functionals on $C(B)$, but also as a family of continuous functions $(K,\|\cdot\|_{L^{\infty}(B)}) \to \R$. Above, we showed that $\Gamma_{\varepsilon}(\psi) \to \Gamma(\psi)$ for every $\psi \in C(B)$: thus, the orbits $\{\Gamma_{\varepsilon}(\psi) : \varepsilon > 0\}$ are bounded for every $\psi \in C(B)$ -- and for every $\psi \in K$, in particular. Applying the Banach-Steinhaus theorem again, we see that
\begin{equation}\label{form32} |\Gamma_{\varepsilon}(\psi)| \leq C\|\psi\|_{L^{\infty}(B)}, \qquad \psi \in C(B), \end{equation}
for some constant $C > 0$ independent of $\varepsilon$. This implies that the functions $\Gamma_{\varepsilon}$ are equicontinuous on $K$: if $\psi \in K$ and $\delta > 0$, we have $|\Gamma_{\varepsilon}(\psi) - \Gamma_{\varepsilon}(\eta)| = |\Gamma_{\varepsilon}(\psi - \eta)| \leq \delta$ as soon as $\|\psi - \eta\|_{L^{\infty}(B)} \leq \delta/C$, for any $0 < \varepsilon \leq 1$. We have now demonstrated that $\{\Gamma_{\varepsilon} \colon K \to \R : 0 < \varepsilon \leq 1\}$ is a pointwise bounded equicontinuous family of functions. By the Arzelà-Ascoli theorem, see \cite[Theorem A5]{Ru}, every sequence in $\{\Gamma_{\varepsilon} : 0 < \varepsilon \leq 1\}$ contains a uniformly convergent subsequence. According to our result on pointwise convergence, the only possible limit of any such sequence is the functional $\psi \mapsto \int \psi \, d\mu_{t}$. This finishes the proof.
\end{proof}

\begin{cor} Let $(\varphi_{\varepsilon})_{\varepsilon > 0}$ be a sequence of smooth test functions satisfying $\varphi_{\varepsilon} \geq \chi_{\varepsilon}$. Let $\mu$ be a compactly supported Radon measure on $\R^{2}$, and write $\tilde{\mu}_{\varepsilon} := \mu \ast \varphi_{\varepsilon}$. Then
\begin{displaymath} \iint_{\R^{2} \times \R^{2}} \eta(x;y) \, d\mu_{t} x \, d\mu_{t} y \leq \liminf_{\varepsilon \to 0} \iint_{\R \times \R} \eta((r,t);(s,t))\tilde{\mu}_{\varepsilon}(r,t)\tilde{\mu}_{\varepsilon}(s,t) \, dr \, ds \end{displaymath}
for all non-negative lower semicontinuous functions $\eta \colon \R^{2} \times \R^{2} \to \R$, and for all $t \in \R$ such that $\mu_{t} = \mu_{\rho_{2},t}$ exists.
\end{cor}

\begin{proof} Assume that $\mu_{t}$ exists. Approximating from below, it suffices to prove the inequality for continuous compactly supported functions $\eta \colon \R^{2} \times \R^{2} \to \R$. For such $\eta$, the plan is to first prove equality with $\tilde{\mu}_{\varepsilon}$ replaced by $\mu_{\varepsilon} = \mu \ast \chi_{\varepsilon}$, and then simply apply the estimate $\mu_{\varepsilon} \leq \tilde{\mu}_{\varepsilon}$.

Fix $\eta \in C(\R^{2} \times \R^{2})$. It follows immediately from the previous lemma that
\begin{displaymath} \iint_{\R^{2} \times \R^{2}} \eta(x;y) \, d\mu_{t}x \, d\mu_{t}y = \int_{\R^{2}} \lim_{\varepsilon \to 0} \int_{\R} \eta((r,t);y)\mu_{\varepsilon}(r,t) \, dr \, d\mu_{t} y. \end{displaymath}
Moreover, the convergence of the inner integrals is uniform in $\epsilon$ on the compact family of functions $K = \{\eta(\cdot \, ; y) : y \in \spt \mu\} \subset C(\R^{2})$. Hence, the numbers 
\begin{displaymath} \left|\int_{\R} \eta((r,t);y)\mu_{\varepsilon}(r,t) \, dr \right|, \qquad y \in \spt \mu, \end{displaymath}
are uniformly bounded by a constant independent of $\epsilon$. This means that the use of the dominated convergence theorem is justified: 
\begin{equation}\label{form31} \iint_{\R^{2} \times \R^{2}} \eta(x;y) \, d\mu_{t}x \, d\mu_{t}y = \lim_{\varepsilon \to 0} \int_{\R^{2}}\int_{\R}\eta((r,t);y)\mu_{\varepsilon}(r,t) \, dr \, d\mu_{t}y. \end{equation}
We may now estimate as follows:
\begin{align} & \left| \iint_{\R^{2} \times \R^{2}} \eta(x;y) \, d\mu_{t}x \, d\mu_{t}y - \iint_{\R \times \R} \eta((r,t);(s,t))\mu_{\varepsilon}(r,t)\mu_{\varepsilon}(s,t) \, dr \, ds \right|\notag\\
& \leq \left| \iint_{\R^{2} \times \R^{2}}\eta(x;y) \, d\mu_{t}x\,d\mu_{t}y - \int_{\R^{2}}\int_{\R} \eta((r,t);y)\mu_{\varepsilon}(r,t) \, dr \, d\mu_{t}y \right|\notag\\
&\label{form33} \qquad + \left|\int_{\R} \left[ \int_{\R^{2}} \eta((r,t);y) \, d\mu_{t}y - \int_{\R} \eta((r,t);(s,t))\mu_{\varepsilon}(s,t) \,ds \right]\mu_{\varepsilon}(r,t) \, dr \right|  \end{align}
As $\varepsilon \to 0$, the first term tends to zero according to \eqref{form31}. To see that the second term vanishes as well, we need to apply the previous lemma again. Namely, we first observe that the outer integration (with respect to $r$) can be restricted to some compact interval $[-R,R]$. The family of continuous mappings $\{\eta((r,t); \cdot) \colon \R^{2} \to \R : r \in [-R,R]\}$ is then compact, so the previous lemma implies that
\begin{displaymath} \int_{\R} \eta((r,t);(s,t))\mu_{\varepsilon}(s,t) \, ds \to \int_{\R^{2}} \eta((r,t);y) \, d\mu_{t} y \end{displaymath}
uniformly with respect to $r \in [-R,R]$, as $\varepsilon \to 0$. An application of \eqref{form32} then shows that the term on line \eqref{form33} converges to zero as $\varepsilon \to 0$. As we mentioned at the beginning of the proof, the assertion of the corollary now follows from the inequality $\mu_{\varepsilon} \leq \tilde{\mu}_{\varepsilon}$.

\end{proof}

The corollary will soon be used to prove an inequality concerning the energies of sliced measures. First, though, let us make a brief summary on Fourier transforms of measures on $\R^{n}$. If $\mu$ is a finite Borel measure on $\R^{n}$, then its Fourier transform $\hat{\mu}$ is, by definition, the complex function
\begin{displaymath} \hat{\mu}(x) = \int_{\R^{n}} e^{-ix \cdot y} \, d\mu y, \qquad x \in \R^{n}. \end{displaymath}
It is well-known, see \cite[Lemma 12.12]{Mat3}, that the $s$-energy $I_{s}(\mu)$ of $\mu$ can be expressed in terms of the Fourier transform:
\begin{equation}\label{s-energy} I_{s}(\mu) := \iint \frac{d\mu x \, d\mu y}{|x - y|^{s}} = c_{s,n}\int_{\R^{n}} |x|^{s - n}|\hat{\mu}(x)|^{2} \, dx, \qquad 0 < s < n. \end{equation}
This will be applied with $n = 1$ below.

\begin{lemma} Let $\mu$ be a compactly supported Radon measure on $\R^{2}$. Then, with $\mu_{t}$ as in the previous lemma, we have
\begin{displaymath} \int_{\R} I_{d - 1}(\mu_{t}) \, dt \lesssim_{d} \int_{\R^{2}} |\rho_{1}(x)|^{d - 2}|\hat{\mu}(x)|^{2} \, dx, \quad 1 < d < 2. \end{displaymath}
\end{lemma}

\begin{proof} Choose a family of test functions $(\varphi_{\varepsilon})_{\varepsilon > 0}$ such that $\chi_{\varepsilon} \leq \varphi_{\varepsilon}$, and $\|\varphi_{\varepsilon}\|_{L^{1}(\R^{2})} \leq 2$. Applying the previous corollary with $\tilde{\mu}_{\varepsilon} = \mu \ast \varphi_{\varepsilon}$ and $\eta(x,y) = |x - y|^{1 - d}$, we estimate
\begin{align*} I_{d - 1}(\mu_{t}) & = \iint_{\R^{2} \times \R^{2}} |x - y|^{1 - d} \, d\mu_{t}x \, d\mu_{t}y\\
& \leq \liminf_{\varepsilon \to 0} \iint_{\R \times \R} |r - s|^{1 - d} \tilde{\mu}_{\varepsilon}(r,t)\tilde{\mu}_{\varepsilon}(s,t) \, dr \, ds.  \end{align*}
Integrating with respect to $t \in \R$, applying (\ref{s-energy}) and using Plancherel yields
\begin{align*} \int_{\R} I_{d - 1}(\mu_{t}) \, dt & \lesssim_{d} \liminf_{\varepsilon \to 0} \int_{\R} \int_{\R} |s|^{d - 2} \left|\int_{\R} e^{irs} \tilde{\mu}_{\varepsilon}(r,t) \, dr \right|^{2} \, ds \, dt\\
& = \liminf_{\varepsilon \to 0}\int_{\R} |s|^{d - 2} \int_{\R} \left| \int_{\R} e^{irs}\tilde{\mu}_{\varepsilon}(r,t) \, dr\right|^{2} \, dt \, ds\\
& \asymp \liminf_{\varepsilon \to 0} \int_{\R} |s|^{d - 2} \int_{\R} \left| \int_{\R} e^{itu} \int_{\R} e^{irs} \tilde{\mu}_{\varepsilon}(r,u) \, dr \, du \right|^{2} \, dt \, ds\\
& = \liminf_{\varepsilon \to 0} \iint_{\R \times \R} |s|^{d - 2} \left|\iint_{\R \times \R} e^{i(s,t) \cdot (r,u)} \tilde{\mu}_{\varepsilon}(r,u) \, dr \, du \right|^{2} \, ds \, dt\\
& = \liminf_{\varepsilon \to 0} \int_{\R^{2}} |\rho_{1}(x)|^{d - 2} |\widehat{\tilde{\mu}_{\varepsilon}}(x)|^{2} \, dx\\
& = \liminf_{\varepsilon \to 0} \int_{\R^{2}} |\rho_{1}(x)|^{d - 2} |\widehat{\varphi_{\varepsilon}}(x)|^{2}|\hat{\mu}(x)|^{2} \, dx\\
& \lesssim \liminf_{\varepsilon \to 0} \int_{\R^{2}} |\rho_{1}(x)|^{d - 2} |\hat{\mu}(x)|^{2} \, dx\\
& = \int_{\R^{2}} |\rho_{1}(x)|^{d - 2} |\hat{\mu}(x)|^{2} \, dx, \end{align*} 
as claimed.
\end{proof}

Our last third lemma concerns the divergent set of certain parametrised power series. The result is due to Peres and Schlag, and the proof can be found in \cite{PS}:

\begin{lemma}\label{series} Let $U \subset \R$ be an open interval, and let $(h_{j})_{j \in \N} \in C^{\infty}(U)$. Suppose that there exist finite constants $B > 1$, $R > 1$, $C > 0$, $L \in \N$ and $C_{l} > 0$, $l \in \{1,\ldots,L\}$, such that $\|h_{j}^{(l)}\|_{\infty} \leq C_{l}B^{jl}$ and $\|h_{j}\|_{1} \leq CR^{-j}$, for $j \in \N$ and $1 \leq l \leq L$. Then, if $1 \leq \tilde{R} < R$ and $\alpha \in (0,1)$ is such that $B^{\alpha}\tilde{R}^{\alpha/L} \leq R/\tilde{R} \leq B\tilde{R}^{1/L}$, we have the estimate
\begin{displaymath} \dim \left\{ \lambda \in U : \sum_{j \in \N} \tilde{R}^{j}|h_{j}(\lambda)| = \infty\right\} \leq 1 - \alpha. \end{displaymath}
\end{lemma} 

\begin{remark} The exact reference to this result is \cite[Lemma 3.1]{PS}. The result there is formulated with $L = \infty$, but the proof actually yields the slightly stronger statement above -- and we will need it. The stronger version is observed by Peres and Schlag themselves on the first few lines of \cite[\S3.2]{PS}.

\end{remark}

\section{The Main Lemma}

Now we are equipped to study the dimension of the sliced measures 
\begin{displaymath} \mu_{\lambda,t} := (\Psi_{\lambda\sharp}\mu)_{\rho_{2},t}, \end{displaymath}
where $\Psi_{\lambda} \colon \Omega \to \R^{2}$ is the function introduced in Definition \ref{psi}. Thus, we are not attempting to slice the measure $\mu$ with respect to the transversal projections $\pi_{\lambda}$. Rather, we first map the measure into $\R^{2}$ with a sufficiently dimension preserving function $\Psi_{\lambda}$, and then slice the image measure $\Psi_{\lambda\sharp}\mu$. The reason for this is simple: in $\Omega$, the residence of $\mu$, we could not use Fourier analytic machinery required to prove Lemma \ref{thm1} below.
\begin{lemma}\label{thm1} Let $\mu$ be a Radon measure on $\Omega$, and let $1 < s < 2$. 
\begin{itemize} 
\item[(i)] If the projections $\pi_{\lambda}$ satisfy the regularity and transversality assumptions (\ref{transversality}) and (\ref{regularity}) with $\tau = 0$, then $I_{s}(\mu) < \infty$ implies
\begin{equation}\label{thm1Ineq} \dim \left\{\lambda \in J : \int_{\R} I_{s - 1}(\mu_{\lambda,t}) \, dt = \infty\right\} \leq 2 - s. \end{equation}

\item[(ii)] If the projections $\pi_{\lambda}$ only satisfy (\ref{transversality}) and (\ref{regularity}) for some $\tau > 0$ so small that $s + \tau^{1/3} < 2$, we still have (\ref{thm1Ineq}), assuming that $I_{t}(\mu) < \infty$ for  $t \geq s + \varepsilon(\tau)$, where $\varepsilon(\tau) > 0$ is a constant depending on $\tau$. Moreover, $\varepsilon(\tau) \to 0$ as $\tau \to 0$.
\end{itemize}
\end{lemma}

\begin{proof} Assume that $I_{t}(\mu) < \infty$ for some $t \geq s$. We aim to determine the range of parameters $\tau > 0$ for which (\ref{thm1Ineq}) holds under this hypothesis. According to our second lemma we have
\begin{displaymath} \int_{\R} I_{s - 1}(\mu_{\lambda,t}) \, dt \lesssim_{s} \int_{\R^{2}} |\rho_{1}(x)|^{s - 2}|\widehat{\Psi_{\lambda\sharp}\mu}(x)|^{2} \, dx \end{displaymath}
for any $\lambda \in J$. We will attempt to show that the integral on the right hand side is finite for as many $\lambda \in J$ as possible. This is achieved by expressing the integral in the form of a power series and then applying Lemma \ref{series}. Some of the work in verifying the conditions of Lemma \ref{series} involves practically replicating ingredients from Peres and Schlag's original proof of Theorem \ref{sobolev}. Instead of using phrases such as 'we then argue as in \cite{PS}', we provide all the details for the reader's convenience, some of them in the Appendix.

Fix a compact subinterval $I \subset J$, and let $\lambda \in I$. We start by splitting our integral in two pieces:
\begin{align}\label{form5} \int_{\R^{2}} |\rho_{1}(x)|^{s - 2}|\widehat{\Psi_{\lambda\sharp}\mu}(x)|^{2} \, dx & =  \int_{\mathcal{C}} |\rho_{1}(x)|^{s - 2}|\widehat{\Psi_{\lambda\sharp}\mu}(x)|^{2} \, dx\\
&\label{form6} + \int_{\R^{2} \setminus \mathcal{C}} |\rho_{1}(x)|^{s - 2}|\widehat{\Psi_{\lambda\sharp}\mu}(x)|^{2} \, dx,
\end{align}
where $\mathcal{C}$ is the \emph{vertical} cone 
\begin{displaymath} \mathcal{C} = \{z : -\pi/4 \leq \arg z \leq \pi/4\} \cup \{\bar{z} : -\pi/4 \leq \arg z \leq \pi/4\}, \end{displaymath}
Here $\arg z$ is shorthand for \emph{signed angle formed by $z$ with the positive $y$-axis}, and $\bar{z}$ refers to complex conjugation. With this choice of $\mathcal{C}$, we have $|\rho_{1}(x)| \gtrsim |x|$ for $x \in \R^{2} \setminus \mathcal{C}$, which means that the integral on line (\ref{form6}) is easily estimated with the aid of equation \eqref{s-energy} and the H\"older bound \eqref{holder}:
\begin{align*} \int_{\R^{2} \setminus \mathcal{C}} |\rho_{1}(x)|^{s - 2}|\widehat{\Psi_{\lambda\sharp}\mu}(x)|^{2} \, dx & \lesssim_{s}	 \int |x|^{s - 2}|\widehat{\Psi_{\lambda\sharp}\mu}(x)|^{2} \, dx\\
& \asymp_{s} \iint |x - y|^{-s} \, d\Psi_{\lambda\sharp}\mu x \, d\Psi_{\lambda\sharp}\mu y\\
& = \iint_{\Omega \times \Omega} |\Psi_{\lambda}(x) - \Psi_{\lambda}(y)|^{-s} \, d\mu x \, d\mu y\\
& \lesssim_{I,\tau} \iint_{\Omega \times \Omega} d(x,y)^{-(1 + \tau)s} \, d\mu x \, d\mu y \end{align*} 
Thus the integral over $\R^{2} \setminus \mathcal{C}$ is finite for \emph{all} $\lambda \in I$, as soon as $(1 + \tau)s \leq t$, which sets the first restriction for the admissible parameters $\tau > 0$.

Write $\mathcal{C} = \mathcal{C}^{1} \cup \mathcal{C}^{2} \cup \mathcal{C}^{3} \cup \mathcal{C}^{4}$, where $\mathcal{C}^{i}$ is the intersection of $\mathcal{C}$ with the $i^{th}$ quadrant in $\R^{2}$. We will show that, if $\tau > 0$ is small enough, then the integral of $|\rho_{1}(x)|^{s - 2}|\widehat{\Psi_{\lambda\sharp}\mu}(x)|^{2}$ over each of these smaller cones can be infinite for parameters $\lambda$ in a set of dimension at most $2 - s$. This will prove the lemma. The treatment of each of the cones $\mathcal{C}^{i}$ is similar, so we restrict attention to $\mathcal{C}^{2}$, and, for simplicity, write $\mathcal{C} := \mathcal{C}^{2} = \{z \in \C : 0 \leq \arg z \leq \pi/4\}$ (thus $\mathcal{C}$ lies in the upper left quadrant of the plane). Further split $\mathcal{C}$ into sub-cones $\mathcal{C}_{i}$, $i = 2,3,...$, where $\mathcal{C}_{i} = \{z : 2^{-i - 1}\pi < \arg z \leq 2^{-i}\pi\}$. The $y$-axis is not covered, but this has no effect on integration:
\begin{align} \int_{\mathcal{C}} |\rho_{1}(x)|^{s - 2}|\widehat{\Psi_{\lambda\sharp}\mu}(x)|^{2} \, dx & = \sum_{i = 2}^{\infty} \int_{\mathcal{C}_{i}} |\rho_{1}(x)|^{s - 2}|\widehat{\Psi_{\lambda\sharp}\mu}(x)|^{2} \, dx\notag\\
&\label{form7} \asymp_{s} \sum_{i = 2}^{\infty} 2^{i(2 - s)} \int_{\mathcal{C}_{i}} |x|^{s - 2} |\widehat{\Psi_{\lambda\sharp}\mu}(x)|^{2} \,dx. \end{align} 
The passage to (\ref{form7}) follows by writing $|\rho_{1}(x)| = |\rho_{1}(x/|x|)|\cdot |x|$ and noting that $\rho_{1}(\zeta) = \sin (\arg \zeta) \asymp \arg \zeta$ for $\zeta \in \mathcal{C} \cap S^{1}$. To prove that (\ref{form7}) is finite for as many $\lambda \in I$ as possible, we need to replace $\chi_{\mathcal{C}_{i}}$ by something smoother. To this end, choose an infinitely differentiable function $\varphi$ on $\R$ satisfying
\begin{displaymath} \chi_{[-1,1]} \leq \varphi \leq \chi_{[-2,2]}. \end{displaymath}
Then let $\varphi_{i}$ be defined by $\varphi_{i}(t) = \varphi(c_{i}t + a_{i})$, where the numbers $a_{i}$ and $c_{i}$ are so chosen that
\begin{displaymath} \chi_{I_{i}} \leq \varphi_{i} \leq \chi_{2I_{i}}, \end{displaymath} 
with $I_{i} = [\pi2^{-i - 1},\pi2^{-i}]$, and $2I_{i}$ denotes the interval with the same mid-point and twice the length as $I_{i}$. Clearly $c_{i} \asymp 2^{i}$. Now
\begin{displaymath} \int_{\mathcal{C}_{i}} |x|^{s - 2} |\widehat{\Psi_{\lambda\sharp}\mu}(x)|^{2} \,dx \leq \int \varphi_{i}(\arg x)|x|^{s - 2}|\widehat{\Psi_{\lambda\sharp}\mu}(x)|^{2} \, dx. \end{displaymath} 
With our eyes fixed on applying Lemma \ref{series}, define
\begin{displaymath} h_{i}(\lambda) := 2^{-i}\int \varphi_{i}(\arg x)|x|^{s - 2}|\widehat{\Psi_{\lambda\sharp}\mu}(x)|^{2} \, dx. \end{displaymath} 
The theorem will be proven by showing that
\begin{equation}\label{form8} \dim \left\{\lambda \in I : \sum_{i = 2}^{\infty} 2^{i(3 - s)}h_{i}(\lambda) = \infty\right\} \leq 2 - s, \end{equation} 
provided that $\tau > 0$ is small. In order to apply Lemma \ref{series}, we now need to estimate both the derivatives and integrals of the functions $h_{i}$.

\subsection{The Integrals} Write $A_{j} := \{x \in \R^{2} : 2^{j - 1} \leq |x| \leq 2^{j}\}$ and $A_{ij} := \{x \in A_{j} : \arg x \in 2I_{i}\}$ for $i = 2,3,\ldots$ and $j \in \Z$. Then, by the choice of $\varphi_{i}$, we have
\begin{equation}\label{form9} h_{i}(\lambda) \lesssim_{s} 2^{-i} \sum_{j \in \Z} 2^{j(s - 2)} \int_{A_{ij}} |\widehat{\Psi_{\lambda\sharp}\mu}(x)|^{2} \, dx. \end{equation} 
Using basic trigonometry, one sees that $x_{1} \asymp -2^{j - i}$ and $x_{2} \asymp 2^{j}$ for $x = (x_{1},x_{2}) \in A_{ij}$. In other words, one may choose an absolute constant $a \geq 1$ such that the sets $A_{ij}$ are covered by the rectangles $Q_{ij} = [-a2^{j - i},-a^{-1}2^{j - i}] \times [a^{-1}2^{j},a2^{j}]$,  see the picture below. 
\begin{figure}[ht!]
\begin{center}
\includegraphics[scale = 1]{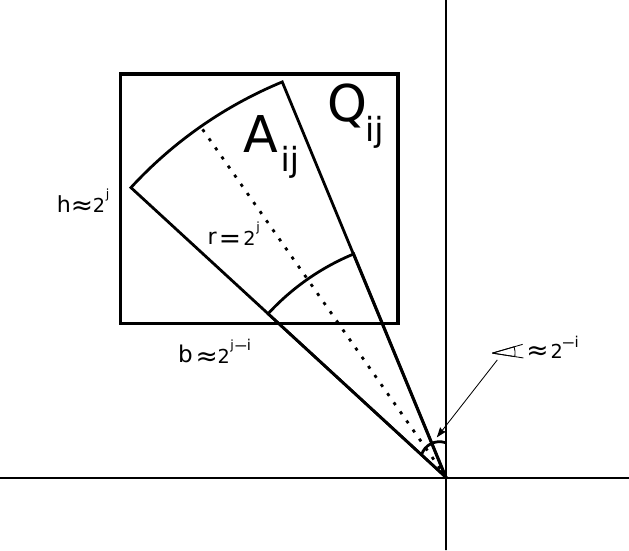}
\end{center}
\caption{The sets $A_{ij}$ and the covering rectangles $Q_{ij}$}
\end{figure}
Next, let $\eta$ be an infinitely differentiable function on $\R$, satisfying $\spt \eta \subset (a^{-2},2a)$ and $\eta|[a^{-1},a] \equiv 1$. Then the function $\eta_{ij} := \widetilde{\eta_{j - i}}\times\eta_{j}$ defined by
\begin{displaymath} \widetilde{\eta_{j - i}}\times\eta_{j}(x) := \widetilde{\eta_{j - i}}(x_{1})\eta_{j}(x_{2}), \qquad x = (x_{1},x_{2}) \in \R^{2}, \end{displaymath}
is identically one on $Q_{ij}$ (and thus $A_{ij}$) for every pair of indices $i = 2,3,\ldots$ and $j \in \Z$, where $\eta_{k}(t) := \eta(2^{-k}t)$ and $\tilde{f}(t) := f(-t)$ for any function $f \colon \R \to \R$. Using some basic properties of the Fourier transform, and the identity 
\begin{displaymath} \Psi_{\lambda}(x) - \Psi_{\lambda}(y) = (d(x,y)\partial_{\lambda}\Phi_{\lambda}(x,y),d(x,y)\Phi_{\lambda}(x,y)), \end{displaymath}
we may now estimate
\begin{align*} \int_{A_{ij}} & |\widehat{\Psi_{\lambda\sharp}\mu}(x)|^{2} \, dx \leq \int \eta_{ij} |\widehat{\Psi_{\lambda\sharp}\mu}(x)|^{2} \, dx = \iint \widehat{\eta_{j - i} \times \widetilde{\eta_{j}}}\,(x - y) \, d\Psi_{\lambda\sharp}\mu x \, d\Psi_{\lambda\sharp}\mu y\\
& = \iint_{\Omega \times \Omega} \widehat{\eta_{j - i}}(\rho_{1}[\Psi_{\lambda}(x) - \Psi_{\lambda}(y)])\widehat{\widetilde{\eta_{j}}}(\rho_{2}[\Psi_{\lambda}(x) - \Psi_{\lambda}(y)]) \, d\mu x \, d\mu y\\
& = 2^{2j - i}\iint_{\Omega \times \Omega} \hat{\eta}(2^{j - i}r\partial_{\lambda}\Phi_{\lambda}(x,y))\bar{\hat{\eta}}(2^{j}(r\Phi_{\lambda}(x,y)) \, d\mu x \, d\mu y, \end{align*} 
where $r := d(x,y)$. To bound the $\lambda$-integral of the expression on the last line, we need
\begin{lemma}\label{PSLemma} Let $\gamma$ be a smooth function supported on $J$. Then, for any $j \in \Z$ and $q \in \N$, we have the estimate
\begin{displaymath} \left| \int_{\R} \gamma(\lambda)\hat{\eta}(2^{j - i}r\partial_{\lambda}\Phi_{\lambda}(x,y))\bar{\hat{\eta}}(2^{j}r\Phi_{\lambda}(x,y)) \, d\lambda \right| \lesssim_{d(\Omega),\gamma,q} (1 + 2^{j}r^{1 + A\tau})^{-q}, \end{displaymath} 
where $A \geq 1$ is some absolute constant.
\end{lemma}  

\begin{proof} The proof given in \cite{PS} for \cite[Lemma 4.6]{PS} extends to our situation. In fact, the statement would be virtually the same as in \cite[Lemma 4.6]{PS} without the presence of the factor 
\begin{displaymath} \hat{\eta}(2^{j - i}r\partial_{\lambda}\Phi_{\lambda}(x,y)). \end{displaymath}
Unfortunately, the proof of \cite[Lemma 4.6]{PS} requires something more delicate than 'bringing the absolute values inside the integral', so this factor cannot be completely dismissed. We discuss the lengthy details in Appendix A. 
\end{proof}

Now we are prepared to estimate the $L^{1}(I)$-norms of the functions $h_{i}$, as required by Lemma \ref{series}. Fix any $r \in (0,1)$, and let $\gamma$ be a smooth function supported on $J$ and identically one on $I$. For brevity, write
\begin{displaymath} \Gamma_{x,y}^{ij}(\lambda) := \gamma(\lambda)\hat{\eta}(2^{j - i}r\partial_{\lambda}\Phi_{\lambda}(x,y))\bar{\hat{\eta}}(2^{j}(r\Phi_{\lambda}(x,y)). \end{displaymath}
Then, use (\ref{form9}) and the Lemma above:
\begin{align*} \sum_{i = 2}^{\infty} 2^{i(r + 1)} \int_{I} h_{i}(\lambda) \, d\lambda & \lesssim_{s} \sum_{i = 2}^{\infty} 2^{ir} \sum_{j \in \Z} 2^{j(s - 2)} \int_{\R} \gamma(\lambda) \int_{A_{ij}} |\widehat{\Psi_{\lambda\sharp}\mu}(x)|^{2} \, dx \, d\lambda\\
& \leq \sum_{i = 2}^{\infty} 2^{i(r - 1)} \sum_{j \in \Z} 2^{js} \iint_{\Omega \times \Omega} \left| \int_{\R} \Gamma_{x,y}^{ij}(\lambda) \, d\lambda \right| \, d\mu x \, d\mu y\\
& \lesssim_{I} \sum_{i = 2}^{\infty} 2^{i(r - 1)} \iint_{\Omega \times \Omega} \sum_{j \in \Z} 2^{js} (1 + 2^{j}d(x,y)^{1 + A\tau})^{-2} \, d\mu x \, d\mu y\\
& \lesssim_{s} \sum_{i = 2}^{\infty} 2^{i(r - 1)} \iint_{\Omega \times \Omega} d(x,y)^{-(1 + A\tau)s} \, d\mu x \, d\mu y
\end{align*}
This sum is finite, if $\tau > 0$ is so small that $(1 + A\tau)s \leq t$. Under this hypothesis, we conclude that
\begin{equation}\label{form10} \|h_{i}\|_{L^{1}(I)} \lesssim_{I,r} 2^{-i(r + 1)}, \qquad r \in (0,1). \end{equation}

\subsection{The Derivatives}
First, we  need to find out the Fourier transform of the function $x \mapsto \varphi_{j}(\arg x)|x|^{s - 2}$. To this end, note that $\varphi$ is certainly smooth enough to have an absolutely convergent Fourier series representation:
\begin{displaymath} \varphi_{j}(\theta) = \sum_{k \in \Z} \hat{\varphi}_{j}(k)e^{ik\theta}. \end{displaymath}
Thus 
\begin{displaymath} \varphi_{j}(\arg x)|x|^{s - 2} = |x|^{s - 2}\sum_{k \in \Z} \hat{\varphi}_{j}(k) e^{ik\arg x} = \sum_{k \in \Z} |x|^{s - 2 - |k|}P_{j,k}(x), \end{displaymath}
where 
\begin{displaymath} P_{j,k}(x) := |x|^{|k|}\hat{\varphi}_{j}(k)e^{ik\arg x} \end{displaymath}
is a harmonic polynomial of degree $|k|$ in $\R^{2}$, which is also homogeneous of degree $|k|$.\footnote{In complex notation, $P_{j,k}(z) = i^{-k}\hat{\varphi}_{j}(k)z^{k}$ if $k \geq 0$, and $P_{j,k} = i^{-k}\hat{\varphi}_{j}(k)\bar{z}^{|k|}$ if $k < 0$. The factor $i^{-k}$ results from the fact that $\arg z = \Arg z - \pi/2$ with our definition of $\arg$.} The Fourier transform of each term $K_{j,k,s}(x) := P_{j,k}(x)|x|^{s - 2 - |k|}$ can be computed by an explicit formula given in \cite[Chapter IV, Theorem 4.1]{SW}:
\begin{displaymath} \hat{K}_{j,k,s}(x) = i^{-|k|}\pi^{1-s}\frac{\Gamma\left(\frac{|k| + s}{2}\right)}{\Gamma\left(\frac{2 + |k| - s}{2}\right)}P_{j,k}(x)|x|^{-|k| - s}.\end{displaymath} 
Thus, the Fourier transform of $x \mapsto \varphi_{j}(\arg x)|x|^{s - 2}$ is the function
\begin{displaymath} F_{j}(x) = \pi^{1-s}|x|^{-s}\sum_{k \in \Z} i^{-|k|}\frac{\Gamma\left(\frac{|k| + s}{2}\right)}{\Gamma\left(\frac{2 + |k| - s}{2}\right)}\hat{\varphi}_{j}(k)H_{k}(x), \end{displaymath}
where $H_{k}(x) = e^{ik\arg x}$. We may now rewrite $h_{j}$ as
\begin{align*} h_{j}(\lambda) & = 2^{-j}\iint F_{j}(x - y) \, d\Psi_{\lambda\sharp}\mu(x) \, d\Psi_{\lambda\sharp}\mu(y)\\
& = 2^{-j}\iint F_{j}(\Psi_{\lambda}(x) - \Psi_{\lambda}(y)) \, d\mu x \, d\mu y, \end{align*}
whence, in order to evaluate the $\partial_{\lambda}^{(l)}$-derivatives of $h_{j}$, we only need to consider the corresponding derivatives of the mappings $\lambda \mapsto F_{j}(\Psi_{\lambda}(x) - \Psi_{\lambda}(y))$ for arbitrary $x,y \in \Omega$. From the bounds we obtain for the derivative, it will be clear that if $I_{t}(\mu) < \infty$ for some $t \geq s$ large enough, then no issues arise from the legitimacy of exchanging the order or differentiation and integration. For fixed $x,y \in \Omega$, the mapping $\lambda \mapsto F_{j}(\Psi_{\lambda}(x) - \Psi_{\lambda}(y))$ is the composition of $F_{j}$ with the path $\lambda \mapsto \Psi_{\lambda}(x) - \Psi_{\lambda}(y) =: \gamma(\lambda)$, whence $\partial_{\lambda}^{(l)}F_{j}(\Psi_{\lambda}(x) - \Psi_{\lambda}(y)) = (F_{j} \circ \gamma)^{(l)}(\lambda)$. 
\begin{proposition}\label{Pderivatives} The derivative $(F_{j} \circ \gamma)^{(l)}(\lambda)$ consists of finitely many terms of the form 
\begin{displaymath} \partial^{\beta}F_{j}(\gamma(\lambda))\tilde{\gamma}_{1}(\lambda)\cdots\tilde{\gamma}_{|\beta|}(\lambda), \end{displaymath}
where $\beta$ is a multi-index of length $|\beta| \leq l$ and $\tilde{\gamma}_{j} = \gamma_{i}^{(k)}$ for some $i \in \{1,2\}$ and $k \leq l$.
\end{proposition}

\begin{proof} To get the induction started, note that 
\begin{displaymath} (F_{j} \circ \gamma)'(\lambda) = \nabla F_{j}(\gamma(\lambda)) \cdot \gamma'(\lambda) = \partial_{1}F_{j}(\gamma(\lambda))\gamma_{1}'(\lambda) + \partial_{2}F_{j}(\gamma(\lambda))\gamma_{2}'(\lambda). \end{displaymath} 
The expression on the right is certainly of the correct form. Then assume that the claim holds up to some $l \geq 1$. Then $(F_{j} \circ \gamma)^{(l + 1)}(\lambda)$ is the sum of the \emph{derivatives} of the terms occurring in the expression for $(F_{j} \circ \gamma)^{(l)}(\lambda)$. Take one such term $\partial^{\beta}F_{j}(\gamma(\lambda))\tilde{\gamma}_{1}(\lambda)\cdots\tilde{\gamma}_{|\beta|}(\lambda)$. The product rule shows that the derivative of this term is
\begin{displaymath} [\nabla \partial^{\beta}F_{j}(\gamma(\lambda)) \cdot \gamma'(\lambda)]\tilde{\gamma}_{1}(\lambda)\cdots\tilde{\gamma}_{|\beta|}(\lambda) + \partial^{\beta}F_{j}(\gamma(\lambda))\sum_{j = 1}^{|\beta|} \tilde{\gamma}_{1}(\lambda) \cdots \tilde{\gamma}_{j}'(\lambda) \cdots \tilde{\gamma}_{|\beta|}(\lambda) \end{displaymath}
The claim now follows immediately from this formula.
\end{proof}
At this point we state the estimate we aim to prove:
\begin{claim}\label{Fderivatives} Let, as before, $\gamma(\lambda) = \gamma_{x,y}(\lambda) = \Psi_{\lambda}(x) - \Psi_{\lambda}(y)$. Then
\begin{displaymath} |(F_{j} \circ \gamma)^{(l)}(\lambda)| \lesssim_{I,l,s} 2^{j(l + 1)}d(x,y)^{-s - \tau(l^{2} + l + s)}, \qquad \lambda \in I,\: l \geq 0. \end{displaymath}
\end{claim}
Once this is established, we will immediately obtain
\begin{displaymath} \|h_{j}^{(l)}\|_{L^{\infty}(I)} \lesssim_{I,l,s} 2^{jl} \iint d(x,y)^{-s - \tau(l(l + 1) + s)} \, d\mu x \, d\mu y = 2^{jl}I_{s + \tau(l(l + 1) + s)}(\mu), \end{displaymath}
so that Lemma \ref{series} can be applied with $B = 2$ and any $L \in \N$ such that
\begin{equation}\label{form11} s + \tau(L(L + 1) + s) \leq t. \end{equation} 
By Proposition \ref{Pderivatives}, it suffices to prove Claim \ref{Fderivatives} for all products of the form $\partial^{\beta} F_{j}(\gamma(\lambda))\tilde{\gamma}_{1}(\lambda)\cdots\tilde{\gamma}_{|\beta|}(\lambda)$, where $|\beta| \leq l$ and $\tilde{\gamma}_{j} = \gamma_{i}^{(k)}$ for some $i \in \{1,2\}$ and $k \leq l$. First of all, 
\begin{displaymath} |\tilde{\gamma}_{j}(\lambda)| \leq |\gamma^{(k)}(\lambda)| = |\partial_{\lambda}^{k}(\Psi_{\lambda}(x) - \Psi_{\lambda}(y))| \lesssim_{I,l,\tau} d(x,y)^{1 - (k + 1)\tau} \lesssim d(x,y)^{1 - (l + 1)\tau} \end{displaymath}
according to the regularity assumption (\ref{regularity}). This yields
\begin{displaymath} |\partial^{\beta} F_{j}(\gamma(\lambda))\tilde{\gamma}_{1}(\lambda)\cdots\tilde{\gamma}_{|\beta|}(\lambda)| \lesssim_{I,l} |\partial^{\beta}F_{j}(\gamma(\lambda))|d(x,y)^{|\beta|(1 - (l + 1)\tau)}. \end{displaymath}
Next we will prove that $|\partial^{\beta}F_{j}(\gamma(\lambda))| \lesssim_{I,l,s} 2^{j(l + 1)}d(x,y)^{-(1 + \tau)(s + |\beta|)}$ for multi-indices $\beta$ of length $|\beta| \leq l$. This will prove Claim \ref{Fderivatives}. One of the factors in the definition of $F_{j}$ is the Riesz kernel $x \mapsto k_{s}(x) = |x|^{-s}$, the expression of which does not depend on $j \geq 1$. It is easily checked that $|\partial^{\beta}k_{s}(x)| \lesssim_{l} s^{|\beta|}|x|^{-s - |\beta|}$, if $|\beta| \leq l$. What about the other factor? First we need the following estimate:
\begin{equation}\label{form12} |\partial^{\beta}H_{k}(x)| \lesssim_{l} |k|^{|\beta|}|x|^{-|\beta|} \leq |k|^{l}|x|^{-|\beta|}, \qquad k \in \Z, \: |\beta| \leq l.  \end{equation}
Note that $H_{k}(z) = i^{-k}z^{k}/|z|^{k}$ or $H_{k}(z) = i^{-k}\bar{z}^{|k|}/|z|^{|k|}$ for $z \in \C = \R^{2}$ (depending on whether $k \geq 0$ or $k < 0$). First,
\begin{displaymath} |\partial^{\beta}z^{k}| = |\partial^{\beta}(x + iy)^{k}| = |k(k - 1)\cdots(k - |\beta| + 1)(x + iy)^{k - |\beta|}| \leq |k|^{|\beta|}|z|^{k - |\beta|}, \end{displaymath}
and the same estimate holds for $\partial^{\beta}\bar{z}^{|k|}$. Second, the estimate for the derivatives of the Riesz kernel yields $|\partial^{\beta}|z|^{-k}| \lesssim_{l} |k|^{|\beta|}|x|^{-k - |\beta|}$ for $k \geq 0$ and $|\beta| \leq l$, and (\ref{form12}) then follows by applying the Leibnitz formula.

Mostly for convenience,\footnote{Any estimate of the form $\Gamma(x + \alpha)/\Gamma(x) \lesssim x^{c(\alpha)}$ would suffice to us.} we use the well-known fact that $\Gamma(x + \alpha)/\Gamma(x) \asymp x^{\alpha}$ for $\alpha > 0$ and large $x > 0$. In particular,
\begin{displaymath} \frac{\Gamma\left(\frac{|k| + s}{2}\right)}{\Gamma\left(\frac{2 + |k| - s}{2}\right)} = \frac{\Gamma\left(\frac{2 + |k| - s}{2} + s - 1\right)}{\Gamma\left(\frac{2 + |k| - s}{2}\right)} \lesssim |k|^{s - 1}, \end{displaymath}
Now we may use the rapid decay bound $|\hat{\varphi}_{j}(k)| = c_{j}^{-1}|\hat{\varphi}(k/c_{j})| \lesssim_{l} 2^{j(l + 1)}|k|^{-(l + 2)}$ to conclude that
\begin{align*} \left| \partial^{\beta} \sum_{k \in \Z} i^{-|k|}\frac{\Gamma\left(\frac{|k| + s}{2}\right)}{\Gamma\left(\frac{2 + |k| - s}{2}\right)}\hat{\varphi}_{j}(k)H_{k}(x) \right| & \lesssim_{l} 2^{j(l + 1)} \sum_{k \in \Z} |k|^{-(l + 2)}|k|^{s - 1}|\partial^{\beta}H_{k}(x)|\\
& \lesssim 2^{j(l + 1)}|x|^{-|\beta|} \sum_{k \in \Z} |k|^{s - 3} \lesssim_{s} 2^{j(l + 1)}|x|^{-|\beta|}. \end{align*} 
Finally, using Leibnitz's rule once more, and also the fact from estimate \eqref{holder} that 
\begin{displaymath} |\gamma(\lambda)| = |\Psi_{\lambda}(x) - \Psi_{\lambda}(y)| \gtrsim_{I,\tau} d(x,y)^{1 + \tau}, \end{displaymath}
we obtain
\begin{displaymath} |\partial^{\beta}F_{j}(\gamma(\lambda))| \lesssim_{l,s} 2^{j(l + 1)}\sum_{\alpha \leq \beta}|c_{\alpha\beta}||\gamma(\lambda)|^{-s - |\alpha|}|\gamma(\lambda)|^{|\alpha| - |\beta|} \lesssim_{l} 2^{j(l + 1)}d(x,y)^{-(1 + \tau)(s + |\beta|)} \end{displaymath}
This finishes the proof of Claim \ref{Fderivatives}.

\subsection{Completion of the Proof of Lemma \ref{thm1}}

Now we are prepared to apply Lemma \ref{series} and prove (\ref{form8}). Let us first consider the simpler case $\tau = 0$. Let $t = s$, $B = 2$ and $\tilde{R} = 2^{3 - s}$. Then fix $\alpha < s - 1$ and choose $r \in (0,1)$ so close to one that $\alpha < r + s - 2$. Next, let $R := 2^{r + 1}$, and choose $L \in \N$ be so large that $\alpha + \alpha(3 - s)/L \leq r + s - 2$. As $\tau = 0$, the hypotheses $(1 + A\tau)s \leq t$ and $s + \tau(L(L + 1) + s) \leq t$ from (\ref{form10}) and (\ref{form11}) are automatically satisfied. With these notations, we have
\begin{displaymath} \|h_{i}\|_{L^{1}(I)} \lesssim R^{-i}, \quad \|h_{i}^{(l)}\|_{L^{\infty}(I)} \lesssim_{l} B^{il} \text{ for } 1 \leq l \leq L, \quad B^{\alpha}\tilde{R}^{\alpha/L} \leq \frac{R}{\tilde{R}} \leq B\tilde{R}^{1/L}, \end{displaymath} 
whence Lemma \ref{series} yields
\begin{displaymath} \dim \left\{\lambda \in I : \sum_{i = 2}^{\infty} 2^{i(3 - s)}h_{i}(\lambda) = \infty\right\} \leq 1 - \alpha. \end{displaymath} 
The proof of Lemma \ref{thm1}(i) is finished by letting $\alpha \nearrow s - 1$. 

In the case (ii) of Lemma \ref{thm1}, we first set $s(\tau) := s + \tau^{1/3} < 2$ and define the functions $\tilde{h}_{i}$ in the same way as the functions $h_{i}$, only replacing $s$ by $s(\tau)$ in the definition. Then we write $\alpha := s - 1 < s(\tau) - 1$ and choose $r := 1 - \tau^{1/3}/2$. If $L$ is an integer satisfying $4\tau^{-1/3} \leq L \leq 8\tau^{-1/3}$, we then have
\begin{displaymath} \alpha + \alpha(3 - s(\tau))/L \leq \alpha + 2/L \leq s - 1 + \tau^{1/3}/2 = r + s(\tau) - 2, \end{displaymath}
which means that
\begin{displaymath} B^{\alpha}\tilde{R}^{\alpha/L} \leq \frac{R}{\tilde{R}} \leq 2 \leq B\tilde{R}^{1/L} \end{displaymath} 
with $B = 2$, $R = 2^{r + 1}$ and $\tilde{R} = 2^{3 - s(\tau)}$. Moreover, the proofs above show (nothing more is required than the 'change of variables' $s \mapsto s(\tau)$) that the estimates
\begin{displaymath} \|\tilde{h}_{i}\|_{L^{1}(I)} \lesssim R^{-i} \quad \text{and} \quad \|\tilde{h}^{(l)}_{i}\|_{L^{\infty}(I)} \lesssim_{l} B^{il} \text{ for } 1 \leq l \leq L, \end{displaymath} 
hold, this time provided that $(1 + A\tau)s(\tau) \leq t$ and $s(\tau) + \tau(L(L + 1) + s(\tau)) \leq t$. The first condition says that $t \geq s + \tau^{1/3} + As\tau + A\tau^{4/3} =: s + \varepsilon_{1}(\tau)$. On the other hand, the upper bound on $L$ gives the estimate
\begin{displaymath} s(\tau) + \tau(L(L + 1) + s(\tau)) \leq s + \tau^{1/3} + 128\tau^{1/3} + \tau(s + \tau^{1/3}) =: s + \varepsilon_{2}(\tau). \end{displaymath} 
Thus, if $t \geq s + \varepsilon(\tau) := s + \max\{\varepsilon_{1}(\tau),\varepsilon_{2}(\tau)\}$, we again have
\begin{displaymath} \dim \left\{\lambda \in I : \sum_{i = 2}^{\infty} 2^{i(3 - s(\tau))}\tilde{h}_{i}(\lambda) = \infty\right\} \leq 1 - \alpha = 2 - s. \end{displaymath}
Arguing in the same way as we arrived at (\ref{form8}), this even proves that
\begin{displaymath} \dim\left\{\lambda \in J : \int_{\R} I_{s(\tau) - 1}(\mu_{\lambda,t}) \, dt = \infty \right\} \leq 2 - s. \end{displaymath} 
The proof of Lemma \ref{thm1}(ii) is finished, since $I_{s - 1}(\mu_{\lambda,t}) \lesssim_{\lambda} I_{s(\tau) - 1}(\mu_{\lambda,t})$ for $\lambda \in J$ and $t \in \R$.
\end{proof}

\section{Proof of Theorem \ref{main} and Applications}

\begin{remark}\label{form3} If $B \subset \Omega$ is a Borel set with $\dim B > 1$, and $\spt \mu \subset B$, the previous lemma can be used to extract information on the dimension of $\Psi_{\lambda}(B) \cap \rho_{2}^{-1}\{t\}$ for various $\lambda \in J$ and $t \in \R$, since $\spt \mu_{\lambda,t} \subset \Psi_{\lambda}(B) \cap \rho_{2}^{-1}\{t\}$. Of course, we were originally interested in knowledge concerning $\dim [B \cap \pi_{\lambda}^{-1}\{t\}]$. Fortunately $\pi_{\lambda} = \rho_{2} \circ \Psi_{\lambda}$, so that
\begin{displaymath} \Psi_{\lambda}(B \cap \pi_{\lambda}^{-1}\{t\}) = \Psi_{\lambda}(B) \cap \rho_{2}^{-1}\{t\}. \end{displaymath}
By the H\"older continuity, recall (\ref{holder}), of the mappings $\Psi_{\lambda}$, we then have
\begin{displaymath} \dim [B \cap \pi_{\lambda}^{-1}\{t\}] \geq (1 - \tau)\dim [\Psi_{\lambda}(B) \cap \rho_{2}^{-1}\{t\}] \end{displaymath}  
for $\lambda \in J$ and $t \in \R$.

\end{remark}

\begin{proof}[Proof of Theorem \ref{main}] Let us first prove the case $\tau = 0$. Assume that $\calH^{s}(B) > 0$. We claim that the set
\begin{displaymath} E :=  \{\lambda \in J : \calL^{1}(\{t \in \R : \dim [B \cap \pi_{\lambda}^{-1}\{t\}] \geq s - 1\}) = 0\} \end{displaymath}
satisfies $\dim E \leq 2 - s$. By Frostman's lemma (the version for metric spaces, see \cite{Ho}), we may choose a non-trivial compactly supported Radon measure $\mu$ on $\Omega$ with $\spt \mu \subset B$ and $\mu(B(x,r)) \leq r^{s}$ for $x \in \Omega$ and $r > 0$. Then $I_{r}(\mu) < \infty$ for every $1 < r < s$, so that by Theorem \ref{sobolev} we have the estimate
\begin{equation}\label{form4} \dim \{\lambda \in J : \pi_{\lambda\sharp}\mu \not\ll \calL^{1}\} \leq 2 - s. \end{equation}
By (\ref{form4}) and Lemma \ref{thm1}, the set
\begin{displaymath} E_{\mu}^{j} := \left\{\lambda \in J : \pi_{\lambda\sharp}\mu \not\ll \calL^{1} \text{ or } \int_{\R} I_{(s - 1/i) - 1}(\mu_{\lambda,t}) \, dt = \infty \text{ for some } i \geq j \right\} \end{displaymath}
has dimension $\dim E_{\mu}^{j} \leq 2 - s + 1/j$. If $\lambda \in J \setminus E_{\mu}^{j}$, we know that the set $N_{\mu,\lambda} := \{t \in \R : \exists \, \mu_{\lambda,t} \text{ and } \mu_{\lambda,t} \not\equiv 0\}$ has positive $\calL^{1}$ measure.\footnote{This follows from the equation $\rho_{2\sharp}(\Psi_{\lambda\sharp}\mu) = \pi_{\lambda\sharp}\mu$: whenever $\pi_{\lambda\sharp}\mu \ll \calL^{1}$, we also have $\rho_{2\sharp}(\Psi_{\lambda\sharp}\mu) \ll \calL^{1}$, and $\calL^{1}(N_{\mu,\lambda}) > 0$ in this case, see Definition \ref{slices}.} Also, by definition of $E_{\mu}^{j}$, we know that for $\calL^{1}$ almost every $t \in N_{\mu,\lambda}$ the measure $\mu_{\lambda,t}$ has finite $(s - 1 - 1/i)$-energy \emph{for every $i \geq j$}: in particular, 
\begin{displaymath} \dim [B \cap \pi_{\lambda}^{-1}\{t\}] \geq \dim [\Psi_{\lambda}(B) \cap \rho_{2}^{-1}\{t\}] \geq \dim (\spt \mu_{\lambda,t}) \geq s - 1 \end{displaymath}
for these $t \in N_{\mu,\lambda}$. Thus $\calL^{1}(\{t \in \R : \dim [B \cap \pi_{\lambda}^{-1}\{t\}] \geq s - 1\}) \geq \calL^{1}(N_{\mu,\lambda}) > 0$, and we conclude that $\lambda \in J \setminus E$. It follows that $E \subset E_{\mu}^{j}$, whence $\dim E \leq 2 - s + 1/j$. Letting $j \to \infty$ now shows that $\dim E \leq 2 - s$, and 
\begin{displaymath} \calL^{1}(\{t \in \R : \dim [B \cap \pi_{\lambda}^{-1}\{t\}] \geq s - 1\}) > 0, \qquad \lambda \in J \setminus E. \end{displaymath}

Next assume that $\tau > 0$, and $\dim B = s$. Let $\delta_{0}(\tau) = 2\varepsilon(\tau)$, where $\varepsilon(\tau)$ is the constant from the previous theorem. Then choose a measure $\mu$ on $\Omega$, supported on $B$, such that $I_{s - \varepsilon(\tau)} < \infty$. Then, as $s - \varepsilon(\tau) = (s - \delta_{0}(\tau)) + \varepsilon(\tau)$, Lemma \ref{thm1} asserts that
\begin{displaymath} \dim \left\{ \lambda \in J : \int_{\R} I_{s - 1 - \delta_{0}(\tau)}(\mu_{\lambda,t}) \, dt = \infty \right\} \leq 2 - s + \delta_{0}(\tau). \end{displaymath} 
On the other hand, Theorem \ref{sobolev} gives
\begin{displaymath} \dim\{\lambda \in J : \pi_{\lambda\sharp}\mu \not\ll \calL^{1}\} \leq 2 - \frac{s - \varepsilon(\tau)}{1 + a_{0}\tau} \leq 2 - s + \frac{\varepsilon(\tau) + 2a_{0}\tau}{1 + a_{0}\tau} =: 2 - s + \delta_{1}(\tau). \end{displaymath} 
This means, again, that the set $N_{\mu,\lambda} = \{t \in \R : \exists \, \mu_{\lambda,t} \text{ and } \mu_{\lambda,t} \not\equiv 0\}$ has positive $\calL^{1}$ measure for all $\lambda \in J$ except for a set of dimension at most $2 - s + \delta_{1}(\tau)$. Writing $\delta(\tau) = \max\{\delta_{0}(\tau) + \tau,\delta_{1}(\tau)\}$, we may conclude that there exists a set $E \subset J$ of dimension at most $2 - s + \delta(\tau)$ such that
\begin{align*} \dim [B \cap \pi_{\lambda}^{-1}\{t\}] & \geq (1 - \tau)\dim [\Psi_{\lambda}(B) \cap \rho_{2}^{-1}\{t\}]\\
& \geq (1 - \tau) \dim (\spt \mu_{\lambda,t})\\
& \geq (1 - \tau)(s - 1 - \delta_{0}(\tau)) \geq s - 1 - \delta(\tau) \end{align*}
holds for $\lambda \in J \setminus E$ and for $\calL^{1}$ almost all $t \in N_{\mu,\lambda}$. In short,
\begin{displaymath} \calL^{1}(\{t \in \R : \dim[B \cap \pi_{\lambda}^{-1}\{t\}] \geq s - 1 - \delta(\tau)\}) = \calL^{1}(N_{\mu,\lambda}) > 0, \quad \lambda \in J \setminus E.\end{displaymath} 
\end{proof}

\begin{remark}\label{sharp} For later use, let us note that the proof above yields the following result: \emph{if $\mu$ is a Borel measure on $\Omega$ with $I_{s}(\mu) < \infty$ for some $1 < s < 2$, then there exists a set $E \subset J$ of dimension $\dim E \leq 2 - s + \delta(\tau)$ such that for every $\lambda \in J \setminus E$ and $\pi_{\lambda\sharp}\mu$ almost every $t \in \R$} it holds that $\pi_{\lambda\sharp}\mu \ll \calL^{1}$ and $\dim [\spt \mu \cap \pi_{\lambda}^{-1}\{t\}] \geq s - 1 - \delta(\tau)$. Indeed, as shown above, the inequality $\dim [\spt \mu \cap \pi_{\lambda}^{-1}\{t\}] \geq s - 1 - \delta(\tau)$ holds for $\calL^{1}$ almost every $t \in N_{\mu,\lambda}$. Moreover, as noted in Definition and Theorem \ref{slices}, we have $\pi_{\lambda\sharp}\mu(\R \setminus N_{\mu,\lambda}) = 0$ whenever $\pi_{\lambda\sharp}\mu \ll \calL^{1}$. 
\end{remark}

\subsection{Applications and Remarks on the Sharpness of Theorem \ref{main}}\label{sharpness}

As the first application of Theorem \ref{main}, we consider orthogonal projections to lines in $\R^{2}$. Namely, let $\pi_{\lambda}(x) := x \cdot (\cos \lambda,\sin \lambda)$ for $x \in \R^{2} = \Omega$ and $\lambda \in (0,2\pi)$. Then $\pi_{\lambda}$ is the orthogonal projection onto the line spanned by the vector $(\cos \lambda, \sin \lambda) \in S^{1}$. These projections are perhaps the most basic example of generalized projections: transversality with $\tau = 0$ follows immdediately from the equation
\begin{displaymath} |\Phi_{\lambda}(x,y)|^{2} + |\partial_{\lambda}\Phi_{\lambda}(x,y)|^{2} = 1, \qquad x,y \in \R^{2}, \end{displaymath}
and the other conditions are even easier to verify. Applied to these projections Theorem \ref{main} assumes the form
\begin{cor} Let $B \subset \R^{2}$ be a Borel set with $0 < \calH^{s}(B) < \infty$ for some $s > 1$. Denote by $L_{\lambda,t}$ the line $L_{\lambda,t} = \{x \in \R^{2} : x \cdot (\cos \lambda,\sin \lambda) = t\} = \pi_{\lambda}^{-1}\{t\}$. Then there exists a set $E \subset (0,2\pi)$ of dimension $\dim E \leq 2 - s$ such that
\begin{displaymath} \calL^{1}(\{t \in \R : \dim [B \cap L_{\lambda,t}] \geq s - 1\}) > 0, \qquad \lambda \in (0,2\pi) \setminus E. \end{displaymath}
\end{cor}

In fact, the inequality $\dim [B \cap L_{\lambda,t}] \geq s - 1$ above could be replaced by equality: indeed, if $\pi \colon \R^{2} \to \R$ is any Lipschitz map, then $\dim [B \cap \pi^{-1}\{t\}] \leq \dim B - 1$ for almost every $t \in \R$. This follows from \cite[Theorem 7.7]{Mat3}.

In the case of orthogonal projections, the bound obtained above for $\dim E$ is sharp. To see this, first note that if $B \subset \R^{2}$ is a Borel set with $\dim B > 1$, and if $\lambda \in (0,2\pi)$ is any parameter such that $\calL^{1}(\{t : \dim [B \cap L_{\lambda,t}] \geq \dim B - 1\}) > 0$, then it follows immediately that $\calL^{1}(\pi_{\lambda}(B)) > 0$. On the other hand, it is known that a compact set $B \subset \R^{2}$ exists with $s = \dim B > 1$, and such that
\begin{equation}\label{FKMP} \dim \{\lambda \in (0,2\pi) : \calL^{1}(\pi_{\lambda}(B)) = 0\} = 2 - s. \end{equation}
This yields
\begin{displaymath} \dim \{\lambda \in (0,2\pi) : \calL^{1}(\{t : \dim [B \cap L_{\lambda,t}] \geq s - 1\}) = 0\} \geq 2 - s \end{displaymath}
for this particular set $B$, so that the bound $\dim E \leq 2 - s$ in the previous corollary cannot be improved. The construction of the compact set $B$ satisfying (\ref{FKMP}) is essentially the same as that of a similar counter-example obtained by Kaufman and Mattila in \cite{KM}. The applicability of the example in \cite{KM} to this situation was observed by K. Falconer in his paper \cite{Fa}, and the full details were worked out by A. Peltomäki in his licenciate thesis \cite{Pe}.

Before discussing another application, we consider the following question. Let $B \subset \R^{2}$ be a Borel set with $\calH^{s}(B) > 0$ for some $s > 1$, and let $(\pi_{\lambda})_{\lambda \in J}$ be a family of projections satisfying the regularity and transversality conditions of Definition \ref{projections} with $\tau = 0$. Then we have the bounds
\begin{equation}\label{exception1} \dim \{\lambda \in J : \calL^{1}(\{t \in \R : \dim [B \cap \pi_{\lambda}^{-1}\{t\}] \geq s - 1\}) = 0\} \leq 2 - s, \end{equation} 
given by Theorem \ref{main}, and
\begin{equation}\label{exception2} \dim \{\lambda \in J : \calL^{1}(\pi_{\lambda}(B)) = 0\} \leq 2 - s, \end{equation}
by Theorem \ref{sobolev}. Moreover, any reader familiar with the proofs of Peres and Schlag in \cite{PS} will note that the projection results there and the slicing result here are very strongly connected. Thus, one might wonder if the 'exceptional sets' of dimension at most $2 - s$ on lines \eqref{exception1} and \eqref{exception2} might, in fact, coincide. In other words the question is, does $\calL^{1}(\pi_{\lambda}(B)) > 0$ imply $\calL^{1}(\{t \in \R : \dim [B \cap \pi_{\lambda}^{-1}\{t\}] \geq s - 1\}) > 0$ (the converse implication is trivial and true, as noted above)? The answer is negative (for orthogonal projections, at least), as can be seen by considering the the self-affine fractal $F \subset [0,1] \times [0,1]$ obtained by iterating the scheme below:

\begin{center}
\includegraphics[scale = 0.9]{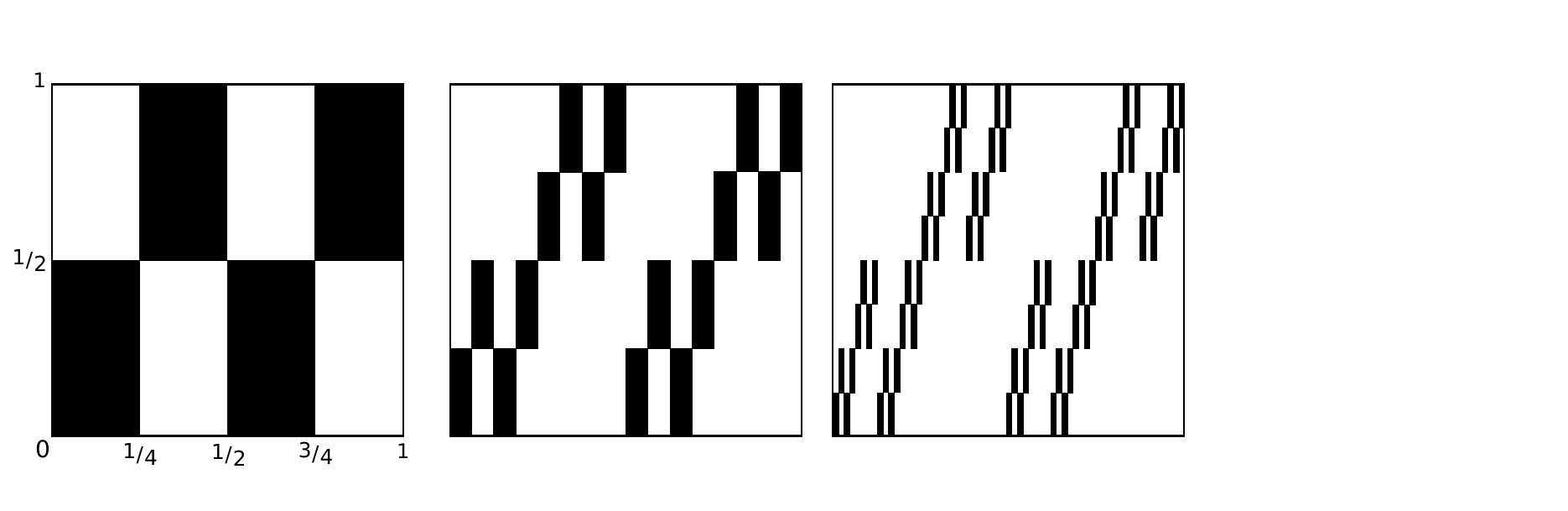}
\end{center}
It is easy to prove that if $\mu$ is the uniformly distributed measure on $F$, then $\mu(Q) \lesssim d(Q)^{3/2}$ for all cubes $Q \subset \R^{2}$, whence $\calH^{3/2}(F) > 0$. Also, the projection of $F$ onto the real axis is the interval $[0,1]$. However, the intersection of any vertical line with $F$ contains at most two points, and, in particular, the dimension of these vertical slices of $F$ have dimension zero instead of $3/2 - 1 = 1/2$.

We finish this section with another application of Theorem \ref{main}. As noted in Remark \ref{ps}, there exist families of projections, which only satisfy the regularity and transversality requirements of Definition \ref{projections} for strictly positive parameters $\tau > 0$. Several such examples are presented in the original paper by Peres and Schlag, so here and now we will content ourselves with just one, the \emph{Bernoulli convolutions}. Let $\Omega = \{-1,1\}^{\N}$ and let $\mu$ be the product measure on $\Omega$. For $\lambda \in (0,1)$, consider the projections
\begin{displaymath} \pi_{\lambda}(\omega) := \sum_{n = 0}^{\infty} \omega_{n}\lambda^{n}, \qquad \omega = (\omega_{1},\omega_{2},\ldots) \in \Omega. \end{displaymath}
Our result yields some information on the number of distinct solutions $\omega \in \Omega$ for the equation $\pi_{\lambda}(\omega) = t$, for $\calL^{1}$ almost every $t \in [-1/(1 - \lambda),1/(1 - \lambda)] =: I_{\lambda}$ (the equation obviously has \emph{no} solutions for $t$ outside $I_{\lambda}$). If $(a,b) \subset (0,1)$, write $G(a,b) := \{\lambda \in (a,b) :  \text{number of solutions to } \pi_{\lambda}(\omega) = t \text{ is uncountable for a.e. } t \in I_{\lambda}\}$.

\begin{cor} If $(a,b) \subset (2^{-1},2^{-2/3}] \approx (0.5,0.63]$, we have
\begin{displaymath} \dim [(a,b) \setminus G(a,b)] \leq 2 - \frac{\log 2}{-\log a}. \end{displaymath} 
\end{cor}

\begin{proof} The proof is essentially the same as the proof of \cite[Theorem 5.4]{PS}, but we provide it here for completeness. We may assume that $a > 1/2$. Fix $\tau > 0$ and divide the interval $(a,b)$ into finitely many subintervals $J_{i} = (a_{i},b_{i})$ satisfying $a_{i} > b_{i}^{1 + \tau}$. Given $J_{i}$, define the metric $d_{i}$ on $\Omega$ by setting $d_{i}(\omega^{1},\omega^{2}) := b_{i}^{|\omega^{1} \land \,\omega^{2}|}$, where $|\omega^{1} \land \,\omega^{2}| = \min\{j \geq 0 : \omega^{1}_{j} \neq \omega^{2}_{j}\}$. As shown in \cite[Lemma 5.3]{PS}, the projections $\pi_{\lambda}$ satisfy the requirements of Definition \ref{projections} on $J_{i}$ with the metric $d_{i}$ and the fixed constant $\tau > 0$. 

Let $\mu$ be the product measure on $\Omega$: then $I_{s}(\mu) < \infty$ with respect to $d_{i}$, if and only if $b_{i}^{s} > 1/2$, if and only if $s  < -\log 2/\log b_{i}$. Let $s_{i} := -\log 2/\log b_{i} - \delta(\tau)$, where $\delta(\tau)$ is the constant from Theorem \ref{main}. Using Remark \ref{sharp}, one then finds a set $E_{i} \subset J_{i}$ of dimension 
\begin{displaymath} \dim E_{i} \leq 2 - s_{i} + \delta(\tau) \leq 2 + \log 2/\log a + 2\delta(\tau) \end{displaymath}
such that $\pi_{\lambda\sharp}\mu \ll \calL^{1}$ and 
\begin{displaymath} \dim \pi_{\lambda}^{-1}\{t\} = \dim [\spt \, \mu \cap \pi_{\lambda}^{-1}\{t\}] \geq s_{i} - \delta(\tau) \end{displaymath}
for every $\lambda \in J_{i} \setminus E_{i}$ and $\pi_{\lambda\sharp}\mu$ almost every $t \in \R$. By now it is well-known, see \cite{MS}, that $\pi_{\lambda\sharp}\mu \ll \calL^{1}$ implies $\calL^{1}\llcorner I_{\lambda} \ll \pi_{\lambda\sharp}\mu$. Thus, we see that the inequality $\dim \pi_{\lambda}^{-1}\{t\} \geq s_{i} - \delta(\tau)$ holds for $\calL^{1}$ almost every $t \in I_{\lambda}$. Also, we clearly have $\card(\{\omega \in \Omega : \pi_{\lambda}(\omega) = t\}) > \aleph_{0}$, whenever $\dim \pi_{\lambda}^{-1}\{t\} \geq s_{i} - \delta(\tau) > 0$, that is, whenever $\tau > 0$ is small enough. Thus $\dim [(a,b) \setminus G(a,b)] \leq 2 + \log 2/\log a + 2\delta(\tau)$. The proof is finished by letting $\tau \to 0$.

\end{proof}

\begin{remark} The proof above does not work with $(2^{-1},2^{-2/3}]$ replaced with $(2^{-1},1]$. This unfortunate fact is concealed in our reference to \cite[Lemma 5.3]{PS} above: the transversality conditions cannot be verified for the projections $\pi_{\lambda}$ on the whole interval $(2^{-1},1]$. Our methods reveal nothing on the dimension of the sets $(a,b) \setminus G(a,b)$ for intervals $(a,b)$ outside $(2^{-1},2^{-2/3})$.

\end{remark}

\subsection{Further Generalisations}\label{further}

In \cite{PS} Peres and Schlag obtained results analogous to Theorem \ref{sobolev} for $\R^{m}$-valued projections with an $\R^{n}$-valued parameter set. It seems reasonable to conjecture that an analogue of Theorem \ref{main} would also hold for such projections, but the technical details might be rather tedious. In some special circumstances, even the inherently one-dimensional Theorem \ref{main} can be used to deduce results for projections with a multidimensional parameter set. For example, let us consider the projections $\pi_{\lambda} \colon \R^{2} \to [0,\infty)$, $\lambda \in \R^{2}$, defined by $\pi_{\lambda}(x) := |\lambda - x|^{2}$. Assume that $B \subset \R^{2}$ is a Borel set with $\calH^{s}(B) > 0$ for some $s > 3/2$.
\begin{claim} There exists a point $\lambda \in B$ such that
\begin{equation}\label{form14} \calL^{1}(\{r > 0 : \dim [B \cap S(\lambda,r)] \geq s - 1\}) > 0, \end{equation} 
where $S(\lambda,r) = \pi_{\lambda}^{-1}\{r^{2}\} = \{y : |\lambda - y| = r\}$. 
\end{claim}

\begin{proof} It follows from the classical line-slicing result of Marstrand (but also from Theorem \ref{main} applied to orthogonal projections) that there exists a line $L \subset \R^{2}$ such that $\dim [B \cap L] > 1/2$. The mappings $\pi_{\lambda}$ with $\lambda \in L \cong \R$ now form a collection of projections with a one-dimensional parameter set. They do not satisfy transversality as projections from the whole plane to $\R$, but this is no issue: choose a compact subset $C \subset B$ with $\calH^{s}(C) > 0$ \emph{lying entirely on one side of, and well separated from, the line $L$.} Then it is an easy exercise to check that the restricted projections $\pi_{\lambda} \colon C \to \R$, $\lambda \in L \cong \R$, satisfy the requirements of Definition \ref{projections} with $\tau = 0$. Thus, according to Theorem \ref{main}, the equation \eqref{form14} holds for all $\lambda \in L \setminus E$, where $\dim E \leq 2 - \dim C < 1/2$. Since $\dim [B \cap L] > 1/2$, we may find a point $\lambda \in [B \cap L] \setminus E$. This finishes the proof.
\end{proof}

\section{Acknowledgments}

I am grateful to my advisor Pertti Mattila for useful comments. I would also like to give many thanks to an anonymous referee for making numerous detailed observations and pointing out several mistakes in the original manuscript. 

\appendix

\section{Proof of Lemma \ref{PSLemma}}

In this section we provide the details for the proof of Lemma \ref{PSLemma}. \emph{The argument is the same as used to prove \cite[Lemma 4.6]{PS}, and we claim no originality on this part}. As mentioned right after Lemma \ref{PSLemma}, the only reason for reviewing the proof here is to make sure that the factor 
\begin{displaymath} \hat{\eta}(2^{j - i}r\partial_{\lambda}\Phi_{\lambda}(x,y))\end{displaymath}
produces no trouble -- and in particular, no dependence on the index $i \in \N$! In this respect, everything depends on an estimate made no earlier than on page 27 of this paper. 

\begin{lemma}\label{Intervals} Fix $x,y \in \Omega$, $x \neq y$, write $r = d(x,y)$, and let $I$ be a compact subinterval of $J$. The set
\begin{displaymath} \{\lambda \in \interior I : |\Phi_{\lambda}(x,y)| < \delta_{I,\tau}r^{\tau}\} \end{displaymath}
can be written as the countable union of disjoint (maximal) open intervals $I_{1},I_{2},\ldots \subset I$. 
\begin{itemize}
\item[(i)] The intervals $I_{j}$ satisfy $\calL^{1}(I_{j}) \leq 2$. Furthermore, if $I_{j}$ and $I$ have no common boundary, then $r^{2\tau} \lesssim_{I,\tau} \calL^{1}(I_{j})$. Thus, in fact, there are only finitely many intervals $I_{j}$.
\item[(ii)] There exist points $\lambda_{j} \in \bar{I}_{j}$, which satisfy: if $\lambda \in \bar{I}_{j}$, then $|\Phi_{\lambda}(x,y)| \geq |\Phi_{\lambda_{j}}(x,y)|$ and $|\Phi_{\lambda}(x,y)| \geq \delta_{I,\tau}r^{\tau}|\lambda - \lambda_{j}|$. Furthermore, there exists a constant $\varepsilon > 0$, depending only on $I$ and $\tau$, with the following properties: \emph{(a)} if $\lambda \in I_{j}$ and $|\Phi_{\lambda}(x,y)| \leq \delta_{I,\tau}r^{\tau}/2$, then $(\lambda - \varepsilon r^{2\tau},\lambda + \varepsilon r^{2\tau}) \cap I \subset I_{j}$, and \emph{(b)} if $|\Phi_{\lambda_{j}}(x,y)| \geq \delta_{I,\tau}r^{\tau}/2$, then $|\Phi_{\lambda}(x,y)| \geq \delta_{I,\tau}r^{\tau}/2$ for all $\lambda \in (\lambda_{j} - \varepsilon r^{2\tau}, \lambda_{j} + \varepsilon r^{2\tau}) \cap I$.
\end{itemize}
\end{lemma}

\begin{proof} Let $J$ be one of the intervals $I_{j}$, see Figure 2. According to the transversality condition (\ref{transversality}), we have
\begin{displaymath} \partial_{\lambda}\Phi_{\lambda}(x,y) \geq \delta_{I,\tau}r^{\tau} \quad \text{or} \quad \partial_{\lambda}\Phi_{\lambda}(x,y) \leq -\delta_{I,\tau}r^{\tau} \end{displaymath}
for all $\lambda \in J$, which means that the mapping $\lambda \mapsto \Phi_{\lambda}(x,y)$ is strictly monotonic on $J$. The first inequality in (i) follows from (\ref{transversality}) via the  mean value theorem: if $[a,b] \subset J$ and $\xi \in (a,b)$ is the point specified by the mean value theorem, we have
\begin{displaymath} \delta_{I,\tau}r^{\tau}(b - a) \leq |\partial_{\lambda}\Phi_{\xi}(x,y)|(b - a) = |\Phi_{b}(x,y) - \Phi_{a}(x,y)| \leq 2\delta_{I,\tau}r^{\tau}. \end{displaymath}
For the second inequality in (i) we apply the regularity condition (\ref{regularity}): write $J = (a,b)$. Since $J$ and $I$ have no common boundary, we have $\{\Phi_{a}(x,y),\Phi_{b}(x,y)\} = \{-\delta_{I,\tau}r^{\tau},\delta_{I,\tau}r^{\tau}\}$. Hence, by (\ref{regularity}) and the mean value theorem,
\begin{displaymath} C_{I,1,\tau}r^{-\tau}\calL^{1}(J) \geq |\partial_{\lambda}\Phi_{\xi}(x,y)|(b - a) = |\Phi_{b}(x,y) - \Phi_{a}(x,y)| = 2\delta_{I,\tau}r^{\tau}. \end{displaymath}

\begin{center}
\begin{figure}[ht!]
\includegraphics[scale = 0.8]{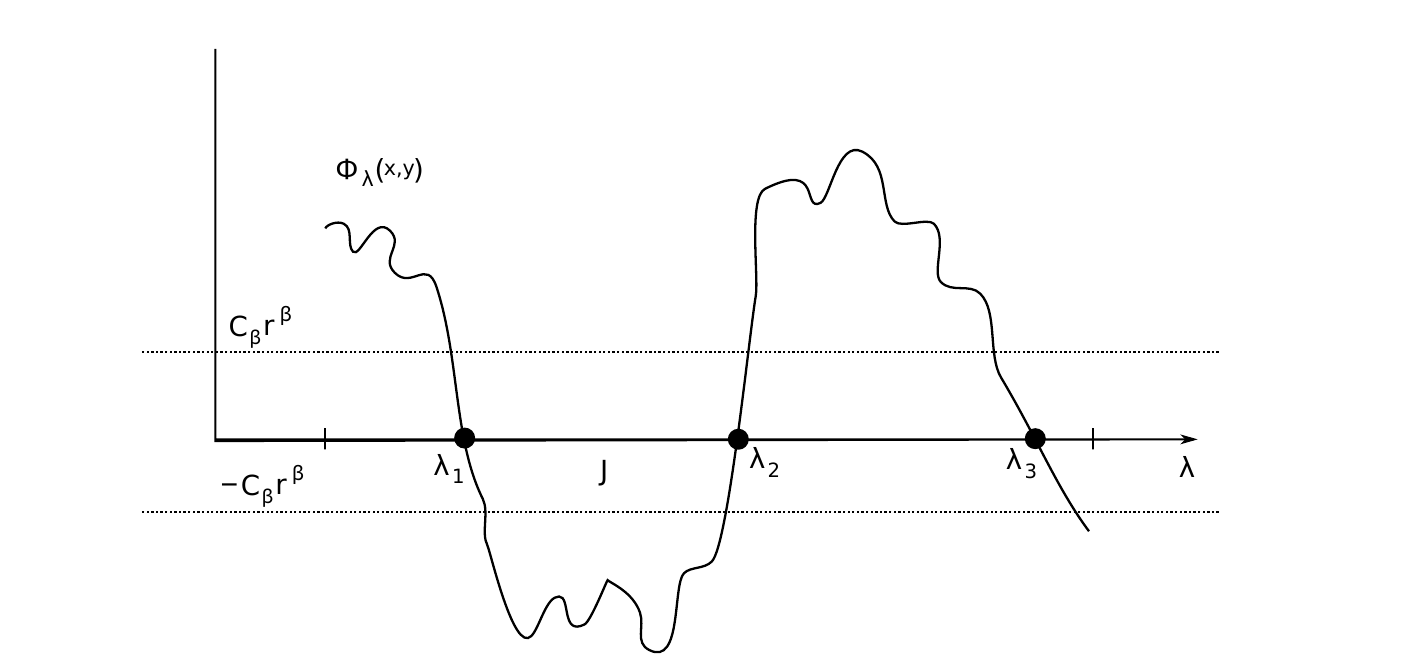}
\caption{The interval $J$ and some of the points $\lambda_{i}$}
\end{figure}
\end{center}
To prove (ii), let $\lambda_{j} \in \bar{I}_{j}$ be the unique point in $\bar{I}_{j}$ where the mapping $\lambda \mapsto |\Phi_{\lambda}(x,y)|$ attains its minimum on $I_{j}$. Such a point exists by continuity and monotonicity: note that on all but possibly two of the intervals $I_{j}$ (the left- and rightmost ones) $\lambda_{j}$ is the unique zero of the mapping $\lambda \mapsto \Phi_{\lambda}(x,y)$ on $I_{j}$. Now if $\lambda \in I_{j}$ is any point, we see that $\Phi_{\lambda}(x,y)$ has the same sign and absolute value at least as great as $\Phi_{\lambda_{j}}(x,y)$. This gives
\begin{displaymath} |\Phi_{\lambda}(x,y)| \geq |\Phi_{\lambda}(x,y) - \Phi_{\lambda_{j}}(x,y)| \geq \delta_{I,\tau}r^{\tau}|\lambda - \lambda_{j}| \end{displaymath}
by \eqref{transversality} and the mean value theorem. All that is left now are (a) and (b) of (iii): set $\varepsilon := \delta_{I,\tau}C_{I,1,\tau}^{-1}/2 > 0$. Assume first that $\lambda \in I_{j}$, $|\Phi_{\lambda}(x,y)| \leq \delta_{I,\tau}r^{\tau}/2$ and $t \in (\lambda - \varepsilon r^{2\tau},\lambda + \varepsilon r^{2\tau}) \cap I$. Then, by the regularity assumption (\ref{regularity}) and $|t - \lambda| < \varepsilon r^{2\tau}$ we get
\begin{displaymath} |\Phi_{t}(x,y)| \leq |\Phi_{t}(x,y) - \Phi_{\lambda}(x,y)| + \delta_{I,\tau}r^{\tau}/2 \leq C_{I,1,\tau}r^{-\tau}|t - \lambda| + \delta_{I,\tau}r^{\tau}/2 < \delta_{I,\tau}r^{\tau}. \end{displaymath}
Since the same also holds for all $t'$ between $\lambda$ and $t$, we see that $t \in I_{j}$. Finally, suppose that $|\Phi_{\lambda_{j}}(x,y)| \geq \delta_{I,\tau}r^{\tau}/2$ and fix $\lambda \in (\lambda_{j} - \varepsilon r^{2\tau}, \lambda_{j} + \varepsilon r^{2\tau}) \cap I$. There are three cases: first, if $\lambda \in I_{j}$, then, by choice of $\lambda_{j}$, we clearly have $|\Phi_{\lambda}(x,y)| \geq \delta_{I,\tau}r^{\tau}/2$. Second, if $\lambda$ belongs to none of the intervals $I_{j}$, we even have the stronger conclusion $|\Phi_{\lambda}(x,y)| \geq \delta_{I,\tau}r^{\tau}$. Finally, if $\lambda \in I_{i}$ for some $i \neq j$, let $t \in I$ be the end point of $I_{i}$ between $\lambda_{j}$ and $\lambda$. Then $|\Phi_{t}(x,y)| = \delta_{I,\tau}r^{\tau}$ and $|t - \lambda| \leq \varepsilon r^{2\tau}$ so that by (\ref{regularity}), the mean value theorem and the choice of $\varepsilon$,
\begin{displaymath} |\Phi_{\lambda}(x,y)| \geq \delta_{I,\tau}r^{\tau} - |\Phi_{t}(x,y) - \Phi_{\lambda}(x,y)| \geq \delta_{I,\tau}r^{\tau} - C_{I,1,\tau}r^{-\tau}\varepsilon r^{2\tau} = \delta_{I,\tau}r^{\tau}/2. \end{displaymath}
This finishes the proof of (ii).
\end{proof}

Now we are prepared to prove Lemma \ref{PSLemma}. Recall that we should prove
\begin{displaymath} \left| \int_{\R} \gamma(\lambda)\hat{\eta}(2^{j - i}r\partial_{\lambda}\Phi_{\lambda}(x,y))\bar{\hat{\eta}}(2^{j}r\Phi_{\lambda}(x,y)) \, d\lambda \right| \lesssim_{\gamma,q} (1 + 2^{j}r^{1 + A\tau})^{-q}, \end{displaymath} 
where $x,y \in \Omega$, $r = d(x,y)$, $q \in \N$, $j \in \Z$, $\gamma$ is any compactly supported smooth function on $J,$ and $A \geq 1$ is an absolute constant. We may and will assume that $q \geq 2$. Moreover, recall that $\eta$ was a fixed smooth function satisfying $\chi_{[a^{-1},a]} \leq \eta \leq \chi_{[a^{-2},2a]}$ for some large $a \geq 1$. Setting $\psi := \bar{\hat{\eta}}$ (in the appendix, we freely recycle all the Greek letters and other symbols that had a special meaning in the previous sections), this definition implies that $\psi$ is a rapidly decreasing function, that is, obeys the bounds $|\psi(t)| \lesssim_{N} (1 + |t|)^{-N}$ for any $N \in \N$, and also satisfies
\begin{equation}\label{form13} \hat{\psi}^{(l)}(0) = 0, \qquad l \in \N. \end{equation}
Fixing $x,y \in \Omega$, we will also temporarily write 
\begin{displaymath} \Gamma(\lambda) := \gamma(\lambda)\hat{\eta}(2^{j - i}r\partial_{\lambda}\Phi_{\lambda}), \end{displaymath}
where $\Phi_{\lambda} := \Phi_{\lambda}(x,y)$. For quite some while, it suffices to know that $\Gamma$ is a smooth function satisfying $\|\Gamma\|_{L^{\infty}(\R)} \lesssim_{\gamma} 1$. The inequality we are supposed to prove now takes the form
\begin{equation}\label{Projection ineq} \left| \int_{\R} \Gamma(\lambda)\psi(2^{j}r\Phi_{\lambda}) \, d\lambda \right| \lesssim_{\gamma,q} (1 + 2^{j}r^{1 + A\tau})^{-q}. \end{equation} 
We claim that $A = 14$ will do the trick. First of all, we may assume that 
\begin{equation}\label{reduction} 2^{j}r^{1 + m\tau} > 1, \qquad 0 \leq m \leq 14, \end{equation} 
Indeed, if this were not the case for some $0 \leq m \leq 14$, then
\begin{displaymath} \left| \int_{\R} \Gamma(\lambda)\psi(2^{j}r\Phi_{\lambda}) \, d\lambda \right| \lesssim \|\gamma\|_{L^{1}(\R)} \lesssim_{\gamma,q} 2^{-q} \leq (1 + 2^{j}r^{1 + m\tau})^{-q} \lesssim_{q,d(\Omega)} (1 + 2^{j}r^{1 + 14\tau})^{-q}, \end{displaymath}
and we would be done. Next choose an auxiliary function $\varphi \in C^{\infty}(\R)$ with $\chi_{[-1/4,1/4]} \leq \varphi \leq \chi_{[-1/2,1/2]}$. Then split the integration in (\ref{Projection ineq}) into two parts:
\begin{align} \int_{\R} \Gamma(\lambda)\psi(& 2^{j}r\Phi_{\lambda}) \, d\lambda \notag\\
&\label{Piece I} = \int_{\R} \Gamma(\lambda)\psi(2^{j}r\Phi_{\lambda})\varphi(\delta_{\tau}^{-1}r^{-\tau}\Phi_{\lambda}) \, d\lambda\\
&\label{Piece II} + \int_{\R} \Gamma(\lambda)\psi(2^{j}r\Phi_{\lambda})[1 - \varphi(\delta_{\tau}^{-1}r^{-\tau}\Phi_{\lambda})] \, d\lambda. \end{align}
Here $\delta_{\tau} := \delta_{I,\tau} > 0$ is the constant from Definition \ref{projections}, and $I \subset J$ is some compact interval containing the support of $\Gamma$. The integral of line (\ref{Piece II}) is easy to bound, since the integrand vanishes whenever $|\Phi_{\lambda}| \leq \delta_{\tau}r^{\tau}/4$, and if $|\Phi_{\lambda}| \geq \delta_{\tau}r^{\tau}/4$, we have the estimate
\begin{displaymath} |\psi(2^{j}r\Phi_{\lambda})| \lesssim_{q} (1 + \delta_{\tau}2^{j - 2}r^{\tau + 1})^{-q} \lesssim_{\tau,q} (1 + 2^{j}r^{\tau + 1})^{-q}, \end{displaymath}
whence
\begin{equation}\label{Rapid decay} \left| \int_{\R} \Gamma(\lambda)\psi(2^{j}r\Phi_{\lambda})[1 - \varphi(\delta_{\tau}^{-1}r^{-\tau}\Phi_{\lambda})] d\lambda \right| \lesssim_{q,\gamma,\tau} (1 + 2^{j}r^{\tau + 1})^{-q}. \end{equation}

Moving on to line \eqref{Piece I}, let the intervals $I_{1},\ldots,I_{N}$, the points $\lambda_{i} \in I_{i}$ and the constant $\varepsilon > 0$ be as provided by Lemma \ref{Intervals} (related to the interval $I$ specified above). Choose another auxiliary function $\chi \in C^{\infty}(\R)$ with $\chi_{[-\varepsilon/2,\varepsilon/2]} \leq \chi \leq \chi_{(-\varepsilon,\varepsilon)}$. Then split the integration on line (\ref{Piece I}) into $N + 1$ parts:
\begin{align} \int_{\R} \Gamma(\lambda)& \psi(2^{j}r\Phi_{\lambda})\varphi(\delta_{\tau}^{-1}r^{-\tau}\Phi_{\lambda}) \, d\lambda \notag\\
&\label{Piece III} = \sum_{i = 1}^{N} \int_{\R} \Gamma(\lambda)\chi(r^{-2\tau}(\lambda - \lambda_{i}))\psi(2^{j}r\Phi_{\lambda})\varphi(\delta_{\tau}^{-1}r^{-\tau}\Phi_{\lambda}) \, d\lambda \\
&\label{Piece IV} + \int_{\R} \Gamma(\lambda) \left[ 1 - \sum_{i = 1}^{N} \chi(r^{-2\tau}(\lambda - \lambda_{i})) \right]\psi(2^{j}r\Phi_{\lambda})\varphi(\delta_{\tau}^{-1}r^{-\tau}\Phi_{\lambda}) \, d\lambda. \end{align}
With the aid of part (ii) of Lemma \ref{Intervals}, the integral on line (\ref{Piece IV}) is easy to handle. If the integrand is non-vanishing at some point $\lambda \in I$, we necessarily have $\varphi(\delta_{\tau}^{-1}r^{-\tau}\Phi_{\lambda}) \neq 0$, which means that $|\Phi_{\lambda}| \leq \delta_{\tau}r^{\tau}/2$: in particular $\lambda \in I_{i}$ for some $1 \leq i \leq N$. Now (a) of Lemma \ref{Intervals}(ii) tells us that $(\lambda - \varepsilon r^{2\tau}, \lambda + \varepsilon r^{2\tau}) \cap I \subset I_{i}$, whence $r^{-2\tau}|\lambda - \lambda_{j}| \geq \varepsilon$ and $\chi(r^{-2\tau}(\lambda - \lambda_{j})) = 0$ for $j \neq i$. But, since the integrand is non-vanishing at $\lambda \in I$, this enables us to conclude that $\chi(r^{-2\tau}(\lambda - \lambda_{i})) < 1$: in particular, $|\lambda - \lambda_{i}| \geq \varepsilon r^{2\tau}/2$. Then Lemma \ref{Intervals}(ii) shows that $|\Phi_{\lambda}| \geq \delta_{\tau}r^{\tau}|\lambda - \lambda_{i}| \geq \varepsilon \delta_{\tau}r^{3\tau}/2$, and using the rapid decay of $\psi$ as on line (\ref{Rapid decay}) one obtains
\begin{displaymath} \left|\int_{\R} \Gamma(\lambda) \left[ 1 - \sum_{j = 1}^{N} \chi(r^{-2\tau}(\lambda - \lambda_{j})) \right]\psi(2^{j}r\Phi_{\lambda})\varphi(\delta_{\tau}^{-1}r^{-\tau}\Phi_{\lambda}) \, d\lambda \right| \lesssim_{q,\gamma,\tau} (1 + 2^{j}r^{1 + 3\tau})^{-q}. \end{displaymath}

Now we turn our attention to the $N$ integrals on line (\ref{Piece III}). If $|\Phi_{\lambda_{i}}| \geq \delta_{\tau}r^{\tau}/2$ for some $1 \leq i \leq N$, part (b) of Lemma \ref{Intervals}(ii) says that $|\Phi_{\lambda}| \geq \delta_{\tau}r^{\tau}/2$ for all $\lambda \in (\lambda_{i} - \varepsilon r^{2\tau},\lambda_{i} + \varepsilon r^{2\tau}) \cap I$, that is, for all such $\lambda \in I$ that $\chi(r^{-2\tau}(\lambda - \lambda_{i})) \neq 0$. But for such $\lambda \in I$ it holds that $\varphi(\delta_{\tau}^{-1}r^{-\tau}\Phi_{\lambda}) = 0$, whence
\begin{displaymath} \int_{\R} \Gamma(\lambda)\chi(r^{-2\tau}(\lambda - \lambda_{i}))\psi(2^{j}r\Phi_{\lambda})\varphi(\delta_{\tau}^{-1}r^{-\tau}\Phi_{\lambda}) \, d\lambda = 0. \end{displaymath}
So we may assume that $|\Phi_{\lambda_{i}}| < \delta_{\tau}r^{\tau}/2$. In this case part (a) of Lemma \ref{Intervals}(ii) tells us that the support of $\lambda \mapsto \Gamma(\lambda)\chi(r^{-2\tau}(\lambda - \lambda_{i}))$ is contained in $I_{i}$. The restriction of $\lambda \mapsto \Phi_{\lambda}$ to this interval $I_{i}$ is strictly monotonic, so the inverse $g := \Phi_{(\cdot)}^{-1} \colon \{\Phi_{\lambda} : \lambda \in I_{i}\} \to I_{i}$ exists. We perform the change of variables $\lambda \mapsto g(u)$:
\begin{align} \int_{\R} & \Gamma(\lambda)\chi(r^{-2\tau}(\lambda - \lambda_{i}))\psi(2^{j}r\Phi_{\lambda})\varphi(\delta_{\tau}^{-1}r^{-\tau}\Phi_{\lambda}) \, d\lambda \notag\\
&\label{Piece IIIb} = \int_{\R} \Gamma(g(u))\chi(r^{-2\tau}(g(u) - \lambda_{i}))\psi(2^{j}r u)\varphi(\delta_{\tau}^{-1}r^{-\tau} u)g'(u) \, du. \end{align}
Write $F(u) := \Gamma(g(u))\chi(r^{-2\tau}(g(u) - \lambda_{i}))\varphi(\delta_{\tau}^{-1}r^{-\tau} u)g'(u)$ for $u \in \{\Phi_{\lambda} : \lambda \in I_{i}\}$. The support of $\lambda \mapsto \Gamma(\lambda)\chi(r^{-2\tau}(\lambda - \lambda_{i}))$ is compactly contained in $I_{i}$, which means that the support of $u \mapsto \Gamma(g(u))\chi(r^{2\tau}(g(u) - \lambda_{i}))$ is compactly contained in the open interval $\{\Phi_{\lambda} : \lambda \in I_{i}\}$. Thus $F$ may be defined smoothly on the real line by setting $F(u) := 0$ for $u \notin \{\Phi_{\lambda} : \lambda \in I_{i}\}$. We will need the following lemmas:

\begin{lemma}\label{Inverse derivatives} Let $h \in C^{k}(a,b)$ with $h'(x) \neq 0$ for $x \in (a,b)$. Suppose the inverse $h^{-1} \colon h(a,b) \to (a,b)$ exists. Then, for $x \in (a,b)$ and $1 \leq l \leq k$,
\begin{displaymath} (h^{-1})^{(l)}(h(x)) = \sum_{m = 0}^{l - 1} \frac{(-1)^{m}}{m!}(h'(x))^{-l - m} \sum_{\mathbf{b} \in \N^{m}} \frac{(l - 1 + m)!}{b_{1}! \cdots b_{m}!}h^{(b_{1})}(x) \cdots h^{(b_{m})}(x), \end{displaymath}
where the inner summation runs over those $\mathbf{b} = (b_{1},\ldots,b_{m}) \in \N^{m}$ with $b_{1} + \ldots + b_{m} = (l - 1) + m$ and $b_{i} \geq 2$ for all $1 \leq i \leq k$.
\end{lemma}

\begin{proof} See \cite{Jo1}. \end{proof}

\begin{lemma}[Faà di Bruno formula] Let $f,h \in C^{l}(\R)$. Then
\begin{equation}\label{Faa di Bruno} (f \circ h)^{(l)}(\lambda) = \sum_{\mathbf{m} \in \N^{l}} b_{\mathbf{m}} f^{(m_{1} + \ldots + m_{l})}(h(\lambda)) \cdot [h'(\lambda)]^{m_{1}} \cdots [h^{(l)}(\lambda)]^{m_{l}}, \end{equation}
where $\mathbf{m} = (m_{1},\ldots,m_{l}) \in \N^{l}$, and $b_{\mathbf{m}} \neq 0$ if and only if $m_{1} + 2m_{2} + \ldots + lm_{l} = l$.
\end{lemma}
\begin{proof} See \cite{Jo2}. \end{proof}

We apply the first lemma to $g$: if $u = \Phi_{\lambda}$ for some $\lambda \in I_{i}$, the derivatives $\partial_{\lambda}^{m}\Phi_{\lambda}$ satisfy the transversality and regularity conditions of Definition \ref{projections}. Hence
\begin{align} |g^{(l)}(u)| & = \left| \sum_{m = 0}^{l - 1} \frac{(-1)^{m}}{m!}(\partial_{\lambda}\Phi_{\lambda})^{-l - m} \sum_{\mathbf{b} \in \N^{m}} \frac{(l - 1 + m)!}{b_{1}! \cdots b_{m}!}\partial^{b_{1}}_{\lambda}\Phi_{\lambda} \cdots \partial^{b_{m}}_{\lambda}\Phi_{\lambda} \right| \notag\\
& \leq \sum_{m = 0}^{l - 1} \sum_{\mathbf{b} \in \N^{m}} \frac{(l - 1 + m)!}{b_{1}! \cdots b_{m}!} (\delta_{\tau}r^{\tau})^{-l - m}(C_{\tau,b_{1}}r^{-\tau b_{1}}) \cdots (C_{\tau,b_{m}}r^{-\tau b_{m}}) \notag\\
& \lesssim_{l,\tau} \sum_{m = 0}^{l - 1} \sum_{\mathbf{b} \in \N^{m}} \frac{(l - 1 + m)!}{b_{1}! \cdots b_{m}!} r^{-\tau(l + m + b_{1} + \ldots + b_{m})} \notag\\
&\label{derivatives} \lesssim_{l,\tau} \sum_{m = 0}^{l - 1} r^{-\tau(l + m + (l - 1) + m)} \lesssim_{l,\tau} r^{-\tau(4l - 3)}, \qquad l \geq 1.  \end{align}
On the last line we simply ignored the summation and employed the inequality
\begin{equation}\label{r ineq} r^{-s} \lesssim_{s,t,d(\Omega)} r^{-t}, \qquad s \leq t. \end{equation} 
with $s = m \leq l - 1 = t$ (this follows from $d(\Omega) < \infty$). Next we wish to estimate the derivatives of $F$, so we recall that
\begin{displaymath} F(u) = \gamma(g(u))\hat{\eta}(2^{j - i}r\partial_{\lambda}\Phi_{g(u)})\chi(r^{-2\tau}(g(u) - \lambda_{i}))\varphi(\delta_{\tau}^{-1}r^{-\tau} u)g'(u)\end{displaymath}
If $u \notin \{\Phi_{\lambda} : \lambda \in I_{i}\}$, we have $F^{(l)}(u) = 0$ for all $l \geq 0$. So, let $u = \Phi_{\lambda}$, $\lambda \in I_{i}$. By (\ref{derivatives}) and (\ref{r ineq}) we have $|g^{(l)}(u)| \lesssim_{l,\tau} r^{-4\tau l}$ for $l \geq 1$, so (\ref{Faa di Bruno}) gives
\begin{align} |(\gamma \circ g)^{(l)}(u)| & \lesssim_{l,\gamma} \sum_{\mathbf{m} \in \N^{l}} |b_{\mathbf{m}}||g'(u)|^{m_{1}} \cdots |g^{(l)}(u)|^{m_{l}} \notag \\
&\label{rho derivative} \lesssim_{l,\tau} \sum_{\mathbf{m} \in \N^{l}} |b_{\mathbf{m}}| (r^{-4 \tau})^{m_{1}} \cdots (r^{-4\tau l})^{m_{l}} \\
& = \sum_{\mathbf{m} \in \N^{l}} |b_{\mathbf{m}}| r^{-4\tau(m_{1} + 2m_{2} + \ldots + lm_{l})} \lesssim_{l} r^{-4\tau l}. \notag \end{align}
A similar computation also yields
\begin{equation}\label{chi and phi derivatives} |\partial^{(l)}_{u}\chi(r^{-2\tau}(g(u) - \lambda_{i}))| \lesssim_{l,\tau} r^{-6\tau l} \quad \text{and} \quad |\partial^{(l)}_{u}\varphi(\delta_{\tau}^{-1}r^{-\tau}u)| \lesssim_{l,\tau} r^{-\tau l}. \end{equation}  
\emph{The presence of the factor $\hat{\eta}(2^{j - i}r\partial_{\lambda}\Phi_{g(u)})$ in the definition of $F(u)$ is the only place where our proof of Lemma \ref{PSLemma} differs from the original proof of \cite[Lemma 4.6]{PS} -- and, indeed, we only need to check that the $l^{th}$ derivatives of this factor admit bounds similar to those of the other factors.} Applying \eqref{Faa di Bruno} with $f(\lambda) = \hat{\eta}(2^{j - i}r\partial_{\lambda}\Phi_{\lambda})$ and $h(u) = g(u)$, we use the bounds $|g^{(l)}(u)| \lesssim_{l,\tau} r^{-4\tau l}$ to obtain
\begin{align*} |\partial^{(l)}_{u} \hat{\eta}(2^{j - i}r\partial_{\lambda}\Phi_{g(u)})| & \leq \sum_{\mathbf{m} \in \N^{l}} |b_{\mathbf{m}}||f^{(m_{1} + \ldots + m_{l})}(g(u))| \cdot |g'(u)|^{m_{1}}\cdots|g^{(l)}(u)|^{m_{l}}\\
& \lesssim_{l,\tau} r^{-4\tau l}\left( \sum_{\mathbf{m} \in \N^{l}} |b_{\mathbf{m}}||f^{(m_{1} + \ldots + m_{l})}(g(u))| \right).
\end{align*}
Here, for any $k = m_{1} + \ldots + m_{l} \leq l$ and $\lambda \in I$, we may use \eqref{Faa di Bruno} again, combined with rapid decay bounds of the form $|\hat{\eta}^{(p)}(t)| \lesssim_{N,p} |t|^{-N}$, to estimate
\begin{align*} |f^{(k)}(\lambda)| & \leq \sum_{\mathbf{m} \in \N^{k}} |b_{\mathbf{m}}||\hat{\eta}^{(m_{1} + \ldots + m_{k})}(2^{j - i}r\partial_{\lambda}\Phi_{\lambda})|(2^{j - i}r)^{m_{1} + \ldots + m_{k}}|\partial_{\lambda}^{2}\Phi_{\lambda}|^{m_{1}}\cdots|\partial_{\lambda}^{k + 1}\Phi_{\lambda}|^{m_{k}}\\
& \lesssim_{l,\tau} \sum_{\mathbf{m} \in \N^{k}} |b_{\mathbf{m}}|\cdot|\partial_{\lambda}\Phi_{\lambda}|^{-(m_{1} + \ldots + m_{k})}\cdot r^{-2m_{1}\tau}\cdots r^{-(k + 1)m_{k}\tau}\\
& \lesssim_{d(\Omega)} \sum_{\mathbf{m} \in \N^{k}} |b_{\mathbf{m}}|\cdot |\partial_{\lambda}\Phi_{\lambda}|^{-(m_{1} + \ldots + m_{k})}\cdot r^{-2k\tau}. \end{align*} 
Finally, recalling that $g(u) \in I_{i}$, we have $|\partial_{\lambda}\Phi_{g(u)}| \gtrsim_{\tau,I} r^{\tau}$ by \eqref{transversality}, whence
\begin{displaymath} |f^{(m_{1} + \ldots + m_{l})}(g(u))| \lesssim_{I,l,\tau} r^{-3(m_{1} + \ldots + m_{l})\tau} \lesssim_{l,\tau} r^{-3\tau l}, \end{displaymath}
and so 
\begin{displaymath} |\partial^{(l)}_{u} \hat{\eta}(2^{j - i}r\partial_{\lambda}\Phi_{g(u)})| \lesssim_{I,l,\tau} r^{-7\tau l}. \end{displaymath} 
Now that we have estimated the derivatives of all the factors of $F$ separately, we may conclude from the Leibnitz formula that
\begin{equation}\label{F derivatives} |F^{(l)}(u)| \lesssim_{\gamma,\tau,l} r^{-7\tau l}, \qquad l \geq 1. \end{equation}
For $l = 0$ estimate (\ref{derivatives}) yields $|F(u)| \lesssim_{\gamma,\tau} |g'(u)| \lesssim_{\tau} r^{-\tau}$.

Next we write $F$ as a degree $2(q - 1)$ Taylor polynomial centered at the origin:
\begin{displaymath} F(u) = \sum_{l = 0}^{2(q - 1)} \frac{F^{(l)}(0)}{l!}u^{l} + \calO(F^{(2q - 1)}(u)u^{2q - 1}). \end{displaymath}
Then we split the integration in (\ref{Piece IIIb}) in two for one last time. Denoting $U_{1} = \{u \in \R : |u| < (2^{j}r)^{-1/2}\}$ and $U_{2} = \R \setminus U_{1}$,
\begin{align}\label{Piece V} \int_{\R} F(u)\psi(2^{j}r u) \, du & = \int_{U_{1}} \psi(2^{j}r u) \left[ \sum_{l = 0}^{2(q - 1)} \frac{F^{(l)}(0)}{l!}u^{l} + \calO(F^{(2q - 1)}(u)u^{2q - 1}) \right] \\
&\label{Piece VI} + \int_{U_{2}} \psi(2^{j}r u)F(u) \, du. \end{align} 
To estimate the integral of line (\ref{Piece VI}) we use the bounds $|\psi(2^{j}ru)| \lesssim_{q} (2^{j}r|u|)^{-2q - 1}$, $|F(u)| \lesssim_{\gamma,\tau} r^{-\tau}$ and $2^{j}r > 1$ to obtain
\begin{equation}\label{Piece VII} \left|\int_{U_{2}} \psi(2^{j}r u)F(u) \, du\right| \lesssim_{q,\gamma,\tau} (2^{j}r)^{-2q - 1} r^{-\tau} \int_{(2^{j}r)^{-1/2}}^{\infty} u^{-2q - 1} \lesssim_{q} (2^{j}r)^{-q}r^{-\tau}.  \end{equation}
As regards (\ref{Piece V}), note that the Fourier transform of $x \mapsto x^{l}\psi(x)$ is (a constant multiple of) $i^{l}\hat{\psi}^{(l)}$, whence
\begin{displaymath} \int_{\R} x^{l}\psi(x) \, dx = \int_{\R} x^{l}\psi(x) e^{-ix \cdot 0} \, dx \asymp i^{l}\hat{\psi}^{(l)}(0) = 0, \quad l \geq 0, \end{displaymath}
according to \eqref{form13}. This shows that
\begin{displaymath} \int_{U_{1}} \psi(2^{j}r u) \sum_{l = 0}^{2(q - 1)} \frac{F^{(l)}(0)}{l!} u^{l} \, du = -\sum_{l = 0}^{2(q - 1)} \int_{U_{2}} \psi(2^{j}r u) \frac{F^{(l)}(0)}{l!} u^{l} \, du. \end{displaymath}
The term corresponding to $l = 0$ may be bounded as on line (\ref{Piece VII}). For $l \geq 1$ we apply the estimates $|\psi(2^{j}ru)| \lesssim_{q} (2^{j}ru)^{-2q - l - 1}$, $1 \leq l \leq q$, and (\ref{F derivatives}):
\begin{align*} \left|\sum_{l = 1}^{2(q - 1)} \int_{U_{2}} \psi(2^{j}r u) \frac{F^{(l)}(0)}{l!} u^{l} \, du \right| & \lesssim_{q,\tau} \sum_{l = 1}^{2(q - 1)} \left( r^{-7\tau l}(2^{j}r)^{-2q - l - 1} \int_{(2^{j}r)^{-1/2}}^{\infty} u^{-2q - 1} \, du \right) \\
& \asymp_{q} \sum_{l = 1}^{2(q - 1)} r^{-7\tau l}(2^{j}r)^{-q - l - 1} \lesssim_{q} r^{-7\tau}(2^{j}r)^{-q}. \end{align*}
On the last line, we need not actually perform the summation: just note that the terms are decreasing, since $2^{j}r^{1 + 7\tau} > 1$, and their number is $2(q - 1) \lesssim_{q} 1$. Since also
\begin{align*} \left|\int_{U_{1}} \psi(2^{j}r u)\calO(F^{(2q - 1)}(u)u^{2q - 1}) \, du\right| & \lesssim_{q,\tau} r^{-7\tau(2q - 1)}\int_{0}^{(2^{j}r)^{-1/2}} u^{2q - 1} \, du \\
& \asymp_{q} r^{-7\tau(2q - 1)}(2^{j}r)^{-q}, \end{align*}
we have shown that
\begin{equation}\label{Piece VIII} \left| \int_{U_{1}} \psi(2^{j}r u) \left[ \sum_{l = 0}^{2(q - 1)} \frac{F^{(l)}(0)}{l!}u^{l} + \calO(F^{(2q - 1)}(u)u^{2q - 1}) \right] \right| \lesssim_{q,\tau} r^{-7\tau(2q - 1)}(2^{j}r)^{-q}. \end{equation}
Now we wrap things up. On line \eqref{Piece III} it follows from part (i) of Lemma \ref{Intervals} that the number $N$ of intervals $I_{i}$ meeting the support of $\Gamma$ (or $\gamma$) is bounded above by $N \lesssim_{\gamma,\tau} r^{-2\tau}$. This, together with the estimates (\ref{Piece VII}) and (\ref{Piece VIII}), yields
\begin{align*} \Bigg| \sum_{i = 1}^{N} \int_{\R} \Gamma(\lambda) & \chi(r^{-2\tau}(\lambda - \lambda_{i}))\psi(2^{j}r\Phi_{\lambda})\varphi(\delta_{\tau}^{-1}r^{-\tau}\Phi_{\lambda}) \, d\lambda \Bigg| \\
& \lesssim_{q,\gamma,\tau,d(\Omega)} r^{-2\tau}((2^{j}r)^{-q}r^{-\tau} + r^{-7\tau(2q - 1)}(2^{j}r)^{-q})\\
& \lesssim_{\tau,d(\Omega)} (2^{j}r^{1 + 14\tau})^{-q} \lesssim_{q} (1 + 2^{j}r^{1 + 14\tau})^{-q}. \end{align*}
The last inequality uses assumption $2^{j}r^{1 + 14\tau} \geq 1$. This finishes the proof, since all the finitely many pieces $I$ into which the integral on line (\ref{Projection ineq}) was decomposed have been seen to satisfy $I \lesssim_{q,\gamma,\tau,d(\Omega)} (1 + 2^{j}r^{1 + c\tau})^{-q}$ for some $0 \leq c \leq 14$.

\end{document}